%% file: main.tex
\RequirePackage{fix-cm}
\documentclass[smallextended]{svjour3}       	
\smartqed  

\usepackage[english]{babel}
\usepackage{amsmath,amssymb,amsfonts,graphicx}
\usepackage[ruled]{algorithm2e} 
\usepackage{amssymb} 
\usepackage{pifont}  

\usepackage[caption=false]{subfig}

\usepackage{color,url}
\usepackage{myshortcuts}

\newtheorem{rmk}[theorem]{\textit{Remark}}
\newtheorem{corollaryb}[theorem]{Corollary}
\renewenvironment{corollary}[1][\unskip]{\begin{corollaryb}}{\end{corollaryb}}

\newtheorem{lemmab}[theorem]{Lemma}
\renewenvironment{lemma}[1][\unskip]{\begin{lemmab}}{\end{lemmab}}

\newtheorem{defn}[theorem]{Definition}

\newtheorem{observation}[theorem]{Observation}

\begin{document}

\title{Restarting for the Tensor Infinite Arnoldi method}


\author{Giampaolo Mele \and Elias Jarlebring }

\institute{
Giampaolo Mele \at
              Dept. Mathematics, KTH Royal Institute of Technology, SeRC swedish e-science research center, Lindstedtsv\"agen 25, Stockholm, Sweden 
              \email{gmele@kth.se}           
                         \and
Elias Jarlebring \at
              Dept. Mathematics, KTH Royal Institute of Technology, SeRC swedish e-science research center, Lindstedtsv\"agen 25, Stockholm, Sweden 
              \email{eliasj@kth.se}           
\date{Received: date / Accepted: date}
}

\maketitle

\input{abstract}
\input{introduction}

\input{expand}

\input{restart_abstract}
\input{explicit_restart}

\input{implicit_restart}

\input{complexity}

\input{numerical_experiments}

\input{conclusion}

\bibliographystyle{elsart-num-sort}
\bibliography{main}

\end{document}

%% file: abstract.tex
\begin{abstract}

An efficient and robust restart strategy is important for any 
Krylov--based method for eigenvalue problems.
The tensor infinite Arnoldi method (TIAR) is 
a Krylov--based method for solving nonlinear eigenvalue problems (NEPs). 
This method can be interpreted as an Arnoldi method applied to a linear and 
infinite dimensional eigenvalue problem where the Krylov basis consists of polynomials.
We propose new restart techniques for TIAR and analyze efficiency and robustness.
More precisely, we consider an extension of TIAR which corresponds to generating the Krylov 
space using not only polynomials but also structured functions that are sums of exponentials and polynomials, while 
maintaining a memory efficient tensor representation.
We propose two restarting strategies, both derived from the specific structure 
of the infinite dimensional Arnoldi factorization. One restarting strategy,
which we call semi--explicit TIAR restart, provides the possibility to carry out locking in a compact way. 
The other strategy, which we call implicit TIAR restart, is based on the Krylov--Schur restart method 
for linear eigenvalue problem and preserves its robustness. 
Both restarting strategies involve approximations of the 
tensor structured factorization in order to reduce complexity and 
required memory resources. We bound the error in the infinite dimensional 
Arnoldi factorization showing that the approximation does not 
substantially influence the robustness of the restart approach.
We illustrate the approaches by applying them to 
solve large scale NEPs that arise 
from a delay differential equation and a wave propagation problem. 
The advantages in comparison to other restart methods are also illustrated.

%
\end{abstract}


%% file: introduction.tex
\section{Introduction}

We consider the \emph{nonlinear eigenvalue problem} (NEP) defined as 
finding $(\lambda,v) \in \CC \times  \CC^n \setminus \left \{ 0 \right \}$ such that
\begin{equation} \label{eq:nep}
 M(\lambda) v = 0 
\end{equation}
where $\lambda \in \Omega \subseteq \CC$, $\Omega$ is 
an open disk centered in the origin and 
$M:\Omega \rightarrow \CC^{n \times n}$ is analytic. 
The NEP has received a considerable attention in literature. See 
the review papers \cite{Mehrmann2004nonlinear,Voss_2013_NEPCHAPTER} and 
the problem collection \cite{NLEP_COLLECTION_2011}. 

There are specialized methods 
for solving different classes of NEPs such as 
polynomial eigenvalue problems (PEPs) see
\cite{mackey2015_PEP,mehrmann2006structured_PEP,lancaster2005pseudospectra} 
and \cite[Chapter 9]{bai2000templates}, 
in particular quadratic eigenvalue problems (QEPs)  
\cite{tisseur2001quadratic,meerbergen2008quadratic,meerbergen2001locking,bai2005soar} 
and rational eigenvalue problems (REPs)
\cite{voss2003maxmin,betcke2008restarting,betcke2016restarting,su2011rep}. 
There are also methods that exploit the structure of the operator $M(\lambda)$ 
like Hermitian structure \cite{szyld2015Hermitian_part1,szyld2015Hermitian_part2} 
or low rank of the matrix--coefficients \cite{van2016rank}. 
Methods for solving a more general class of NEP are also present in literature. 
There exist methods based on modification of the Arnoldi method \cite{voss2004arnoldi}, 
which can be restarted for certain problems, 
Jacobi--Davidson methods \cite{betcke2004jacobi_davidson}, 
Newton--like methods \cite{kressner2009blockNewton,neumaier1985residual,effenberger2013robust}
and Arnoldi--like methods combined with a companion linearization of $M(\lambda)$ 
\cite{guttel2014nleigs,Beeumen_CORK_2015,Jarlebring_INFARN_2012}. 

We do not assume any particular structure of the NEP except for the analyticity 
and the computability of certain quantities associated with $M(\lambda)$ 
(further described later). 
In this paper we consider 
the Infinite Arnoldi method (IAR) \cite{Jarlebring_INFARN_2012}, which  
is equivalent to the Arnoldi method applied to a linear operator. 
More precisely, under the assumption that zero is not an eigenvalue, the problem \eqref{eq:nep} 
can be reformulated as $\lambda B(\lambda) v = v$, where 
$ 
B(\lambda) = 
M(0)^{-1} (M(0)-M(\lambda))/\lambda
$. This problem is equivalent to the linear and 
infinite dimensional eigenvalue problem 
$\lambda \BBB \psi(\theta) = \psi(\theta)$, where 
$\psi(\theta): \CC \rightarrow \CC$ is an analytic 
function \cite[Theorem 3]{Jarlebring_INFARN_2012}. 
The operator  
$\BBB$ is linear, maps functions to functions, and is 
defined as
\begin{align*}
 \BBB \psi(\theta) := 
 \int_0^{\theta} \psi(\hat \theta) d \hat \theta + 
 C(\psi),
\end{align*}
where
\begin{align*}
  C(\psi) 
  := \sum_{i=0}^{\infty} 
  \frac{B^{(i)}(0)}{i!} \psi^{(i)}(0).
\end{align*}

The Tensor Infinite Arnoldi (TIAR), which is an improvement of IAR, 
was presented in \cite{WAVEGUIDE_ARNOLDI_2015}.
This method is equivalent to IAR but computationally more attractive. In contrast to IAR, 
the basis of the Krylov space, which consists of polynomials, is 
implicitly represented in a memory efficient way. 
This improves the performances in terms of memory and
CPU-time. 
Another improvement of IAR was presented in \cite{Jarlebring2014Schur}. 
This method consists in generating 
the Krylov space by using  
structured functions, which are sums of polynomials and exponential functions. 
The main advantage of this approach is the possibility to perform 
a semi--explicit restart by imposing the structure. 
In this paper extend the framework of TIAR to structured functions and 
study restart techniques.

A problematic aspect of any algorithm based on the Arnoldi method is that, 
when many iterations are performed, there are
numerical complexity and stability issues. 
Fortunately, an appropriate restart of the algorithm can partially 
resolve these issues. 
There exist two main classes of restarting strategies: 
explicit restart and 
implicit restart. 
Most of the explicit restart techniques consist in 
selecting a starting vector that 
generates an Arnoldi factorization with the wanted Ritz values.
The implicit restart consists computing
 a new Arnoldi factorization 
with the wanted Ritz values. 
This process can be done 
deflating the unwanted Ritz values as in, e.g., IRA \cite{Sorensen_IRA_1996}
or extracting a proper 
subspace of the Krylov space by using the Krylov--Schur restart approach \cite{Stewart2002KrylovSchur}. 
Both approaches are mathematically 
equivalent. 
For reasons of numerical stability it is in general 
preferable to use implicit restart. 
See \cite{Morgan96RestartArnoldi} for further discussions 
about the restart of the Arnoldi method for 
the linear eigenvalue problems. 

The paper is organized as follows: in Section \ref{sec:expand} 
we extend TIAR to tensor structured function. 
In Section \ref{sec:explicit_restart} we propose 
a semi--explicit restart for TIAR. 
This new algorithm is equivalent to \cite{Jarlebring2014Schur} but the 
 implicit representation of the Krylov basis 
gives an improvements in terms of memory and CPU time. 
In section \ref{sec:implicit_restart} we propose an implicit restart for TIAR based on an 
adaption of Krylov--Schur restart. 
The Krylov--Schur restart for the Arnoldi method in the linear case has constant 
CPU--time for outer iteration. In contrast to this, 
a direct usage of the Krylov--Schur restart for TIAR does not give 
a substantial improvement due to the memory efficient representation of the Krylov basis. 
We show that the structure of 
the Arnoldi factorization 
allow us to perform  
approximations that reduce the complexity 
and the memory requirements. 
We prove that the 
coefficients matrix representing the basis of the Kylov space
present a fast decay in the singular values. Therefore 
we use this in a derivation of a low rank approximation of such matrices. 
Moreover we prove that there is a fast 
decay in the coefficients of the polynomial part of 
the functions in the Krylov space. This 
can be used to introduce another approximation 
when the power series coefficients of $M(\lambda)$ decay to zero                                  . 
We give explicit bounds on the errors due to those approximations.

There exist other 
Arnoldi--like methods combined with a companion linearization that use 
memory efficient representation of the Krylov basis matrix. 
See CORK \cite{Beeumen_CORK_2015}, TOAR \cite{Kressner_and_Roman_compact_2014} and \cite{zhang2013memory}. 
Similar to TIAR, the direct usage of the Krylov--Schur restart for these methods 
does not decrease the complexity unless svd--based approximations are used. 
More precisely, 
the coefficients that represent the Krylov basis are replaced with their low rank approximations. 
In contrast to those approaches, our specific setting allow us to characterize 
the impact of the approximations.

Finally, in Section~\ref{sec:numerical_experiments} 
we show, with numerical simulations, the effectiveness 
of the restarting strategies.

%% file: expand.tex
\section{Tensor structured functions and TIAR factorizations} \label{sec:expand}
Similar to many restart strategies for linear eigenvalue problems,
our approach is based on computation, representation and manipulation
of an Arnoldi-type factorization. For our infinite-dimensional
operator, the analogous Arnoldi-type factorization
is defined as follows. The functions $\Psi_k$ are
represented with a particular tensor structure which we 
further described in Section~\ref{sect:representation}.

\begin{defn}[TIAR factorization]
Let $\Psi_{k+1}(\theta)$ be a tensor 
structured \\ function 
with orthogonal columns and let
$\Hul_{k} \in \CC^{(k+1) \times k}$ be 
an Hessenberg matrix 
with positive elements in the 
sub--diagonal. 
The pair $(\Psi_{k+1},\Hul_{k})$ is a TIAR factorization 
of length $k$ if 
\begin{equation} \label{eq:TIAR_factorization}
\BBB \Psi_{k}(\theta) = \Psi_{k+1}(\theta) \Hul_{k}.
\end{equation}
\end{defn}

\subsection{Representation and properties of the tensor structred functions}\label{sect:representation}

We consider a class of structured functions  introduced in \cite{Jarlebring2014Schur},
represented in a different and memory--efficient way.

\begin{defn}
The vector--valued function $\psi : \CC \rightarrow \CC^n$ is a 
tensor structured function if it exist
$Y, W \in \CC^{n \times p}$, 
$\bar a\in \CC^{d\times r}$, 
$\bar b\in \CC^{d\times p}$, 
$\bar c \in \CC^{p}$,
$S \in \CC^{p \times p}$, 
$Z \in \CC^{n \times r}$ 
where $[Z, \ W]$ is orthogonal 
and $\sspan(Y)=\sspan(W)$, such that 
\begin{align}\label{eq:structured_functions} 
 \psi(\theta)  = &  P_{d-1}(\theta)
\left(
\sum_{\ell=1}^r 
\bar a_{:,\ell} 
\otimes z_\ell 
+ \sum_{\ell=1}^p 
\bar b_{:,\ell}
\otimes w_\ell \right) 
  + Y \exp_{d-1} (\theta S) \bar c 
\end{align}
where
\begin{equation}\label{eq:Pd}
P_d(\theta) :=  (1, \theta, \dots, \theta^d) \otimes I_n 
\end{equation} 
and 
$
\exp_{d-1} (\theta S) 
:= 
\sum_{i=d}^\infty \theta^i S^i
$
is consistent with \cite{Jarlebring2014Schur}.
\end{defn}
The matrix--valued functions 
$\Psi_{k}: \CC \rightarrow \CC^{n \times k}$ 
is a tensor structured function if it can be expressed as
$\Psi_{k}(\theta)=(\psi_1(\theta), \dots, \psi_{k}(\theta))$, where 
each $\psi_i$ is a tensor structured function. 
We denote the $i$--th column of $\Psi_{k}$ by
$\psi_i$. The structure induced by \eqref{eq:structured_functions} 
is now, in a compact form 
\begin{align}\label{eq:structured_functions_block}
 \Psi_{k}(\theta)  = & P_{d-1}(\theta)
\left(
\sum_{\ell=1}^r 
a_{:,:,\ell}
\otimes z_\ell   
+ \sum_{\ell=1}^p 
b_{:,:,\ell} 
\otimes w_\ell \right) 
+ Y \exp_{d-1} (\theta S) C 
\end{align}
where
$a \in \CC^{d \times k \times r}$, 
$b \in \CC^{d \times k \times p}$, 
$C \in \CC^{p \times k}$. 
We say that $\Psi_{k}(\theta)$ 
is orthogonal if the columns are orthogonormal, i.e.,  
$<\psi_i(\theta), \psi_j(\theta)> = \delta_{i,j}$ for 
$i,j=1,\dots,k$. 
We use the scalar product consistent 
with the other papers about the infinite Arnoldi method 
\cite{Jarlebring2014Schur,Jarlebring_INFARN_2012}, i.e.,
if 
$\psi(\theta)=\sum_{i=0}^\infty \theta^ix_i$ and 
$\varphi(\theta)=\sum_{i=0}^\infty \theta^iy_i$, then
\[
<\psi,\varphi>=\sum_{i=0}^\infty <x_i,y_i>.
\] 
The computation of this scalar product 
and norms for the tensor structured functions 
\eqref{eq:structured_functions} can be 
done analogous to \cite{Jarlebring2014Schur}. 
In particular, by definition of \eqref{eq:Pd} 
we have
\begin{equation}  \label{eq:Pd_frob}
  \|P_{d-1}(\theta)W\|=\|W\|_F
\end{equation}
for any $W\in\CC^{nd\times p}$.

\begin{rmk}[Representation of tensor structured functions]
 The polynomial part of a tensor structured function \eqref{eq:structured_functions_block} 
 is a linear combination of the columns of the matrices $Z$ and $W$ 
 using the coefficients 
 given by the tensors $a$ and $b$. 
 The exponential part is given by a linear combination of the columns 
 of the matrix $Y$ and using as coefficients the matrix $C$ 
 multiplied by the powers of $S$. Therefore we can 
 represent a tensor structured function \eqref{eq:structured_functions_block} 
 using the matrices $(Z,W,Y,S)$ and the coefficients 
 $(a,b,C)$.
\end{rmk}

\begin{observation}[Linearity with respect the coefficients] \label{obs:linearity}
Given the tensor structured function $\Psi_k(\theta)$ 
represented by $(Z,W,Y,S)$ with coefficients $(a,b,C)$ and 
$\tilde \Psi_{\tilde k}(\theta)$ represented 
by the same matrices but with coefficients $(\tilde a, \tilde b, \tilde C)$. 
The function $\hat \Psi_{\hat k}(\theta) = \Psi_k(\theta)M + \tilde \Psi_{\tilde k}(\theta) N$ is also 
a tensor structured function 
represented by the same matrices and coefficients
$(\hat a, \hat b, \hat C)$
where for $\ell=1, \dots, r$ we have defined
\begin{align*}
 \hat a_{:,:,\ell} :=  a_{:,:,\ell}M	+ \tilde a_{:,:,\ell}N	&&
 \hat b_{:,:,\ell} :=  b_{:,:,\ell}M	+ \tilde b_{:,:,\ell}N	&&
 \hat C := CM+\tilde CN
\end{align*}

\end{observation}

We use the following notation 
$M_i:=M^{(i)}(0)$ and 
$\MM_d(Y,S)$ is defined as in \cite{Jarlebring2014Schur}.
In particular, any nonlinear function $M$ can  be represented as a sum of products of scalar nonlinearities
\begin{equation} \label{eq:nep_form_f}
M(\lambda) = \sum_{i=1}^q T_i f_i(\lambda),\; \; \;\; T_i\in\CC^{n\times n}, f_i:\Omega\rightarrow\CC,\;i=1,\ldots,q,
\end{equation} 
and we define  $\MM_d:\CC^{n\times p} \times \CC^{p\times p}\rightarrow\CC^{n\times p}$ 
as
\begin{equation} \label{eq:nep1}
\MM_d(Y,S) := \sum_{i=1}^q F_i Y f_i(S) - \sum_{i=1}^{d} \frac{M_i Y S^i}{i!} ,
\end{equation}
which equivalently can be expressed as
\begin{equation} \label{eq:nep2}
 \MM_d(Y,S) = \sum_{i=d+1}^{\infty} \frac{M_i Y S^i}{i!}.
\end{equation}

The action of the operator $\BBB$
 on functions represented as in \eqref{eq:structured_functions} 
can now be expressed in a closed form using the notation above.

\begin{theorem}[Action of $\BBB$] \label{thm:action_of_B}
Suppose 
$Y,W \in \CC^{n \times p}$, 
$Z \in \CC^{n \times r}$,
$\bar a \in \CC^{d \times r}$,
$\bar b \in \CC^{d \times p}$,
$\bar c \in \CC^{p}$
and $S \in \CC^{p \times p}$.
Suppose $\lambda(S) \subset \Omega$, let 
$\tilde c = S^{-1} \bar c$
and 
\begin{equation} \label{eq:z_tilde}
  \tilde z:	= 
  -M_0^{-1} \left[ \MM_d(Y,S)\tilde c - 
  \sum_{i=1}^d M_i \left( \sum_{i=1}^r \frac{\bar a_{i,\ell}}{i} z_\ell + 
  \sum_{\ell=1}^p \frac{\bar b_{i,\ell}}{i} w_\ell\right) \right].
\end{equation}
Under the assumption that
\begin{equation} \label{eq:assumption_z}
\tilde z \not \in \sspan(z_1, \dots, z_r, w_1,\dots, w_p),
\end{equation}
let $z_{r+1}$ be the normalized 
orthogonal complement of $\tilde z$ 
against $z_1, \dots, z_r$, $w_1, \dots, w_p$ 
and $\tilde a_{1,\ell}$ and $\tilde b_{1,\ell}$ be the orthonormalization coefficients, i.e.,
\begin{equation} \label{eq:z_orth}
  \tilde z = \sum_{i=1}^{r+1} \tilde a_{1,\ell} z_i + \sum_{i=1}^{p} \tilde b_{1,\ell} w_i.
\end{equation}

Then, the action of $\BBB$ on 
the tensor structured function defined by 
\eqref{eq:structured_functions} 
is
\begin{equation}
\label{eq:action_of_B}
\BBB \psi (\theta) =
 P_{d}(\theta)
 \left(
\sum_{\ell=1}^{r+1} 
\tilde a_{:,\ell} 
\otimes z_\ell 
+ \sum_{\ell=1}^p 
\tilde b_{i,\ell}
\otimes w_\ell \right) 
  + Y \exp_{d} (\theta S) \tilde c  
\end{equation} 
where 
\begin{subequations} \label{eq:abc_B}
\begin{align}
& \tilde a_{i,r+1} := 0,			& &  i=1, \dots, d \				 					\\
& \tilde a_{i+1,\ell} := \bar a_{i,\ell}/i, 	& &  i=1, \dots, d \ ; \ \ell = 1, \dots, r		\label{eq:abc_B_a}	 	\\
& \tilde b_{i+1,\ell} := \bar b_{i,\ell}/i,	& &  i=1, \dots, d \ ; \ \ell = 1, \dots, p.	
\end{align}
\end{subequations}

\end{theorem}

\begin{proof}
 With the notation
 \begin{align} \label{eq:notation_x}
  x_i := \sum_{\ell=1}^r \bar a_{i+1,\ell} z_\ell + \sum_{\ell=1}^p \bar b_{i+1,\ell} w_\ell
  & & i=0, \dots, d-1
 \end{align}
 and $x:=\vect(x_0, \dots, x_{d-1}) \in \CC^{dn}$, 
 $\psi(\theta)$ defined in \eqref{eq:structured_functions} 
 can be expressed as
 \begin{equation} \label{eq:structure_function_x}
  \psi(\theta) = P_{d-1}(\theta) x + Y \exp_{d-1} (\theta S) \bar c
 \end{equation}
 By invoking \cite[theorem 4.2]{Jarlebring2014Schur} and 
 using \eqref{eq:notation_x}, we can express the action of the operator as
 \begin{equation} \label{eq:B_act_vec}
  \BBB \psi(\theta) = P_{d}(\theta) x_+ + Y \exp_{d} (\theta S) \tilde c
 \end{equation}
 where $x_+:=\vect(x_{+,0}, \dots, x_{+,d}) \in \CC^{(d+1)n}$ with
 \begin{align}
  x_{+,i} &:= \sum_{i=1}^r \frac{\bar a_{i,\ell}}{i} z_\ell + \sum_{\ell=1}^p \frac{\bar b_{i,\ell}}{i} w_\ell
  \label{eq:x+}
  & & i = 1, \dots, d
  \\
  x_{+,0} &:= -M_0^{-1} \left( \MM_d(Y,S)\tilde c + \sum_{i=1}^d M_i x_{+,i} \right)
  \label{eq:x0+}.
 \end{align}
 Substituting \eqref{eq:x+} in \eqref{eq:x0+} we obtain  
 $x_{+,0} = \tilde z $ given in \eqref{eq:z_tilde}. 
 Using \eqref{eq:abc_B} and \eqref{eq:z_orth} 
 we can express $x_+$ in terms of $\tilde a$ and $\tilde b$ and 
 we conclude by substituting this expression for $x_+$ in \eqref{eq:B_act_vec}.
\end{proof}

\begin{rmk}
The assumption \eqref{eq:assumption_z} can only be satisfied if $r+p \le n$. This is the case 
that we are considering in this paper, since we assume the NEP to be large--scale and in Section~\ref{sec:svd_compression}
we introduce approximations that avoid $r$ from being large.
The hypothesis $\lambda(S) \subseteq \Omega$ 
is necessary in order to define $\MM_d(Y,S)$ that is used to 
compute $\tilde z$ 
in equation \eqref{eq:z_tilde}.
\end{rmk}

\subsection{Orthogonalization}

In order to expand a TIAR factorization $(\Psi_{k},\Hul_{k-1})$, 
we need to orthogonalize
the tensor structured function 
$\BBB \psi_{k}$ 
(computed using the theorem \ref{thm:action_of_B})
 against the columns of $\Psi_{k}(\theta)$.
The degree of $\Psi_{k}(\theta)$ is $d-1$
whereas the degree of $\BBB \psi_{k}(\theta)$ 
is $d$. In order to perform the orthogonalization,
 we transform them 
to the same degree $d$. 
Starting from \eqref{eq:structured_functions_block} we can rewrite 
$\Psi_{k}$ as 
\begin{equation} \label{eq:exanding_poly_a}
 \Psi_{k}(\theta)  =  
 P_{d-1}(\theta)
 \left(
\sum_{\ell=1}^r 
a_{:,:,\ell} 
\otimes z_\ell   
+ \sum_{\ell=1}^p 
b_{:,:,\ell}
\otimes w_\ell \right) +
 \frac{Y S^{d}C}{d!}
 \theta^d
+ Y \exp_{d} (\theta S) C.
\end{equation}
We define 
\begin{equation} \label{eq:exanding_poly}
\begin{array}{ccc} \displaystyle
E:=\frac{W^H Y S^{d}C}{d!} & \hspace{1cm}	&
\begin{array}{clcl}
 a_{d,j,\ell}	&:= 0		& \hspace{1cm} &	 \ell=1,\dots, r+1	\\
 b_{d,j,\ell} 	&:= e_{\ell, j}	& \hspace{1cm} &	 \ell=1,\dots, p
\end{array}
\end{array}
\end{equation}
for $j=1,\dots, k$. 
Since $\sspan(W)=\sspan(Y)$ and, 
since $W$ is orthogonal, we have that $Y=WW^H Y$.
Hence, using this relation and \eqref{eq:exanding_poly}, 
the function $\Psi_{k}$ in 
\eqref{eq:exanding_poly_a} can be expressed as 
\begin{align*}
 \Psi_k(\theta)  = 
 P_{d}(\theta)
 \left(
\sum_{\ell=1}^r 
 a_{:,:,\ell} 
\otimes z_\ell   
+ \sum_{\ell=1}^p 
 b_{:,:,\ell} 
\otimes w_\ell \right) 
+ Y \exp_{d} (\theta S) C 
\end{align*}

\begin{theorem}[Orthogonalization]\label{thm:orthogonalization}
Let 
$(Z,W,Y,S)\in\CC^{n\times r}\times\CC^{n\times p}\times\CC^{n\times p}\times\CC^{p\times p}$ be the matrices 
and 
$(a,b,C), (\bar a, \bar b, \bar c) \in\CC^{d \times \bar k \times r}\times \CC^{d \times \bar k \times p}\times\CC^{p \times \bar k}$ 
the coefficients 
that represent 
$\psi(\theta)$ given in \eqref{eq:structured_functions} 
and $\Psi_k(\theta)$ 
given in \eqref{eq:structured_functions_block}. 
Let 
\begin{equation} \label{eq:h}
 h = \sum_{\ell=1}^r 
 (a_{:,:,\ell})^H 
 \bar a_{:,\ell} +
  \sum_{\ell=1}^p 
 (b_{:,:,\ell})^H 
 \bar b_{:,\ell}
 +
 \sum_{i=d}^{\infty} C^H \frac{(S^i)^H Y^H Y S^i}{(i!)^2} \bar c
\end{equation}
The orthogonal complement of 
$\psi(\theta)$ 
against the columns of $\Psi_k(\theta)$ is 
\begin{equation*}
 \psi^\perp(\theta)  =   
P_{d-1}(\theta)
\left(
\sum_{\ell=1}^r 
a_{:,\ell}^\perp
\otimes z_\ell 
+ \sum_{\ell=1}^p 
b_{:,\ell}^\perp
\otimes w_\ell \right) 
  + Y \exp_{d-1} (\theta S) c^\perp 
\end{equation*}
where  
\begin{subequations} \label{eq:abc_orth}
\begin{align} 
 c^\perp		&= 	 \bar c-C h						 			\\
 a_{:,\ell}^\perp 	&=	 \bar a_{:,\ell}- a_{:,:,\ell} h & & \ell=1, \dots, r 	\label{eq:abc_orth_a}	\\
 b_{:,\ell}^\perp 	&=	 \bar b_{:,\ell}- b_{:,:,\ell} h & & \ell=1, \dots, p			
\end{align}
\end{subequations}
The vector $h$ contains the orthogonalization 
coefficients, i.e., $h_j = <\psi_i, \psi>$.
Moreover, given
\begin{equation} \label{eq:beta}
 \beta: = \sqrt{\normF{b^\perp}^2+\normF{a^\perp}^2
+ \sum_{i=d}^{\infty} \frac{ \left(c^\perp\right)^H \left(S^i\right)^H Y^H Y S^i c^\perp}{(i!)^2}}
\end{equation}
it holds $\| \psi^\perp \| = \beta$.
\end{theorem}

\begin{proof}
Let us define $h_j:=<\psi_j, \psi>$ for $j=1, \dots, k$, we have that the orthogonal 
complement, computed with the Gram--Schmidt process, is  
$\psi^\perp(\theta) = \psi(\theta) - \Psi_k(\theta) h$. 
Using the Observation~\ref{obs:linearity}
we obtain directly 
\eqref{eq:abc_orth}. 

We express $\psi(\theta)$ as \eqref{eq:structure_function_x} and, 
the columns of $\Psi_k$ as
\begin{equation} 
  \psi_j(\theta) = P_{d-1}(\theta) x^{(j)} + Y \exp_{d-1} (\theta S) c_j 
\end{equation}
where $x^{(j)}:=\vect(x_0^{(j)}, \dots, x_{d-1}^{(j)}) \in \CC^{dn}$, with 
 \begin{align} \label{eq:notation_x_j}
  x_i^{(j)} := \sum_{i=1}^r \bar a_{i+1,j,\ell} z_\ell + \sum_{\ell=1}^p \bar b_{i+1,j,\ell} w_\ell
  & & i=0, \dots, d-1.
 \end{align}
By applying \cite[equation (4.32)]{Jarlebring2014Schur} we obtain
\begin{align} \label{eq:hj}
 h_j = \sum_{i=0}^{d-1} (x_i^{(j)})^H x_i + 
 c_j^H \sum_{i=d}^{\infty} \frac{(S^i)^H Y^H Y S^i}{(i!)^2} \bar c 
 &&
 j=1, \dots, k.
\end{align}
We now substitute 
\eqref{eq:notation_x} and \eqref{eq:notation_x_j} in \eqref{eq:hj} 
and use the orthogonormality of the vectors 
$z_1, \dots, z_r$, $w_1, \dots, w_p$ and we find that
\begin{align*} 
 h_j = \sum_{\ell=1}^r 
 (a_{:,j,\ell})^H 
 \bar a_{:,\ell} +
  \sum_{\ell=1}^p 
 (b_{:,j,\ell})^H 
 \bar b_{:,\ell}
 +
 \sum_{i=d}^{\infty} c_j^H \frac{(S^i)^H Y^H Y S^i}{(i!)^2} \bar c
  &&
 j=1, \dots, k.
\end{align*}
Which are the elements of the right--hand side of obtain \eqref{eq:h}. 
Using that $ \| \psi^\perp \|^2 = < \psi^\perp, \psi^\perp > $ 
and repeating the same reasoning we have
\begin{align*} 
  \| \psi^\perp \|^2 &= 
  \sum_{\ell=1}^r 
 (a^\perp_{:,\ell})^H 
 a^\perp_{:,\ell} +
  \sum_{\ell=1}^p 
 (b^\perp_{:,\ell})^H 
 b^\perp_{:,\ell}
+ \sum_{i=d}^{\infty} \frac{ \left(c^\perp\right)^H \left(S^i\right)^H Y^H Y S^i c^\perp}{(i!)^2}.
\end{align*}
which proves  \eqref{eq:beta}.

\end{proof}

\subsection{A TIAR expansion algorithm in finite dimension}
One algorithmic component common in many restart procedures is the expansion of 
an Arnoldi-type factorizations. 
The standard way to expand Arnoldi-type factorizations 
(as, e.g., described in \cite[Section 3]{Stewart2002KrylovSchur})
involves the computation of the action of the operator/matrix and 
orthogonalization. We now show how we can carry out an expansion
of the infinite dimensional TIAR-factorization \eqref{eq:TIAR_factorization} 
by only using operations on matrices and
vectors of finite dimension.

In the previous subsections we presented
the action of the operator $\BBB$ and orthogonalization
for tensor structured functions \eqref{eq:structured_functions}.
These results can be directly combined to expand the TIAR factorization.
The resulting algorithm is summarized in 
Algorithm~\ref{alg:expand_TIAR}.
The action of the operator $\BBB$ described in 
Theorem~\ref{thm:action_of_B} is expressed
in Steps~\ref{step:tildez}-\ref{step:tildea}.
The orthogonalization of the new function
using Theorem~\ref{thm:orthogonalization}
is expressed in Steps~\ref{step:h}-\ref{step:H} and
Step~\ref{step:increased} corresponds to 
increasing the degree as described in \eqref{eq:exanding_poly_a}
and \eqref{eq:exanding_poly}. Due to the representation of $\Psi_k$ as
tensor structured function, 
the expansion with one column corresponds to an expansion
 of all the coefficients representing $\Psi_k$. 
This expansion is visualized in Figure~\ref{fig:graph_expand}.

\input{algorithm_expand_TIAR}

\begin{figure} 
\begin{center}
 \includegraphics{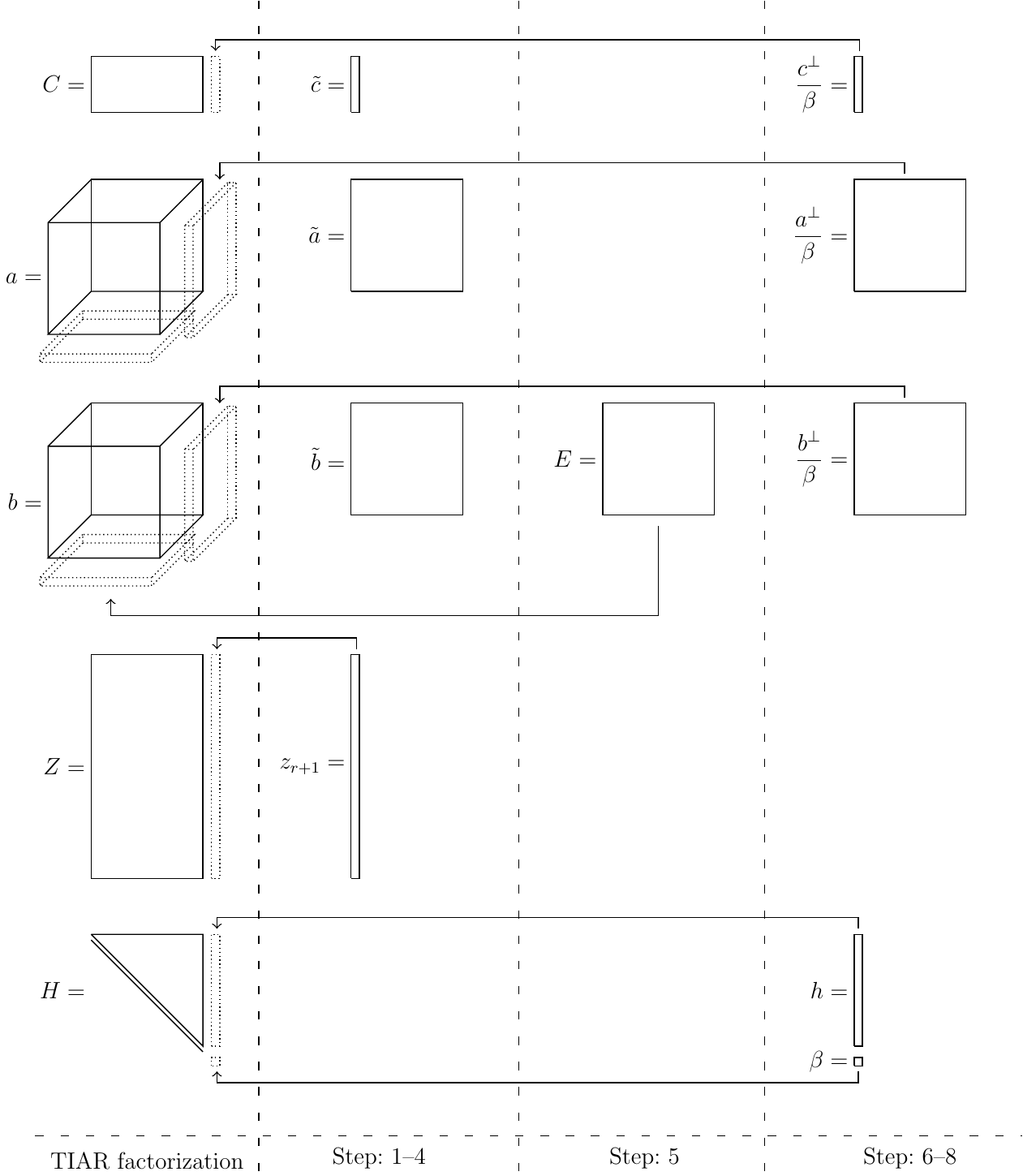} 
\end{center}
\caption{Graphical illustration of the expansion of the tensor structred function that  
represents the TIAR factorization in Algorithm~\ref{alg:expand_TIAR}.\label{fig:graph_expand} }
\end{figure}

%% file: algorithm_expand_TIAR.tex
\begin{algorithm}\label{alg:expand_TIAR}
\caption{Expand TIAR factorization (tensor--structured functions)}
\SetKwInOut{Input}{input}\SetKwInOut{Output}{output}
\Input{
A TIAR factorization $(\Psi_{\bar k+1}, \Hul_{\bar k})$ 
represented by
$(Z,W,Y,S)\in\CC^{n\times r}\times\CC^{n\times p}\times\CC^{n\times p}\times\CC^{p\times p}$
and 
$(a,b,C)\in\CC^{d \times \bar k \times r}\times \CC^{d \times \bar k \times p}\times\CC^{p \times \bar k}$.
}
\Output{
A TIAR factorization $(\Psi_{m+1}, \Hul_{m})$ 
represented by
$(Z,W,Y,S)\in\CC^{n\times \tilde{r}}\times\CC^{n\times p}\times\CC^{n\times p}\times\CC^{p\times p}$
and 
$(a,b,C)\in\CC^{\tilde{d} \times m \times r}\times \CC^{\tilde{d} \times m \times p}\times\CC^{p \times m}$
where $\tilde{r}=r+m-\bar{k}$ and $\tilde{d}=d+m-\bar{k}$.
}
\BlankLine
\nl Set $\tilde{r}=r$, $\tilde{d}=d$\\
\For{$k= \bar k +1,2,\ldots, m$ }{
\nl Compute $\tilde z$ using \eqref{eq:z_tilde}, where 
    $\bar a = a_{:,:,k}$, $\bar b = b_{:,:,k}$ and $\bar c =  c_k$ \label{step:tildez} \\
\nl Compute $z_{\tilde{r}+1}$ and increase $\tilde{r}=\tilde{r}+1$ \\
\nl Set $\tilde a$, $\tilde b$ and $\tilde c$ as in \eqref{eq:abc_B}  \label{step:tildea}\\
\nl Compute $E$ and expand the tensors $a$ and $b$ as \eqref{eq:exanding_poly} 
    and increase $\tilde{d}=\tilde{d}+1$ \label{step:increased}\\
\nl Compute $h$ using \eqref{eq:h}, where $\bar a = \tilde a$, \label{step:h}
    $\bar b = \tilde b$ and $\bar c = \tilde c$\\
\nl Compute $a^\perp, b^\perp, c^\perp$ using \eqref{eq:abc_orth} 
    and $\beta$ using \eqref{eq:beta} \label{step:beta_alg}
and extend 
\[
 \underline{H}_{k} = 
 \begin{pmatrix}
  \underline{H}_{k-1}	&	h	\\
	  0		&	\beta
 \end{pmatrix}
 \in \CC^{(k+1) \times k}
\] \label{step:H}
\\
\nl Expand $ c_{k+1} := c^\perp/\beta$ and 
    $a_{:,k+1,:} := a^\perp/\beta$ and $b_{:,k+1,:} := b^\perp/\beta$. \label{step:divide_beta_alg}
\\
}
\end{algorithm}

%% file: restart_abstract.tex
\section{Restarting for TIAR in an abstract setting}

\subsection{The Krylov-Schur decomposition for TIAR-factorizations}\label{sec:KS}

We briefly recall the reasoning for the Krylov--Schur type restarting \cite{Stewart2002KrylovSchur} 
in an abstract and infinite dimensional setting. We later show that the operations can be carried out with operations
on matrices and vectors of finite size.
Let $(\Psi_{m+1}, \Hul_m)$ be a TIAR factiorization. 
Let $P$ such that 
$P^H H_m P$ is triangular (ordered Schur factorization), then
\begin{equation} \label{eq:Arnoldi_factorization_a}
   \BBB \hat \Psi_m = \hat \Psi_{m+1}  
   \begin{pmatrix} 
    R_{1,1}	&	R_{1,2}	&	R_{1,3}	\\
		&	R_{2,2} &	R_{2,3}	\\
		&		&	R_{3,3}	\\
	a_1^H	&	a_2^H	&	a_3^H
   \end{pmatrix}
\end{equation}
where $\hat \Psi_{m+1} = \left[ \Psi_{m} P, \ \psi_{m+1} \right]$. 
The matrix $P$ is selected in a way that the matrix $R_{1,1} \in \CC^{\pl \times \pl}$ contains the converged 
Ritz values, the matrix $R_{2,2} \in \CC^{(p-\pl) \times (p-\pl)}$ 
contains the wanted Ritz values and the matrix 
$R_{3,3} \in \CC^{(m-p) \times (m-p)}$ contains the 
Ritz values that we want to purge.

From~\eqref{eq:Arnoldi_factorization_a} we find that 
\begin{equation} \label{eq:Arnoldi_factorization_b}
   \BBB \tilde \Psi_p = \tilde \Psi_{p+1}  
   \begin{pmatrix}
    R_{1,1}	&	R_{1,2} \\
		&	R_{2,2}	\\
	a_1^H	&	a_2^H
   \end{pmatrix}
\end{equation}
where $\tilde \Psi_{p+1} := [ \hat \Psi_{m} I_{m+1,p}, \ \psi_{m+1} ] = [\hat \Psi_p, \ \psi_{m+1}]$.

Using a composition of Householder reflections, 
we compute a matrix $Q$ such that 
\begin{equation} \label{eq:Arnoldi_factorization_d}
   \BBB \bar \Psi_p = \bar \Psi_{p+1}  
   \begin{pmatrix} 
    R_{1,1}	&	F			\\
		&	H			\\
	a_1^H	&	\beta e_{p-\pl}^H
   \end{pmatrix}
\end{equation}
where $\bar \Psi_{p+1} = \tilde \Psi_{p+1} [Q \ e_{m+1}] =  [\tilde \Psi_{p} Q \ \psi_{m+1}]$.

Since we want to lock the Ritz values in the matrix $R_{1,1}$,
we replace in \eqref{eq:Arnoldi_factorization_d} 
the vector $a_1$ with zeros, introducing an error $\OOO (\| a_1 \|)$.
With this approximation, \eqref{eq:Arnoldi_factorization_d} is the wanted TIAR factiorization of length $p$. 

\begin{observation}
 In the TIAR factorization \eqref{eq:Arnoldi_factorization_d}, $(\bar \Psi_{\pl}, R_{1,1})$ 
 is an invariant pair, i.e., 
 $\BBB \bar \Psi_{\pl} = \bar \Psi_{\pl} R_{1,1}$. Moreover 
 $(\bar \Psi_{\pl}(0), R_{1,1}^{-1})$ 
 is invariant of the original NEP 
 in the sense of \cite[Definition 1]{kressner2009blockNewton}, see 
 \cite[Theorem 2.2]{Jarlebring2014Schur}.
\end{observation}

\subsection{Two structured restarting approaches}

The standard restart approach for TIAR using Krylov-Schur type restarting,
as described in the previous section, involves expansions and manipulations of the TIAR factorization.
Due linearity of tensor structured functions described in Observation~\ref{obs:linearity}, 
the manipulations for $\Psi_{m}$ leading to $\Psi_p$ can be directly 
carried out on the coefficients representing $\Psi_m$. Unfortunately,
due to the implicit representation of $\Psi_m$, the memory requirements
are not substantially reduced since the basis matrix $Z\in\CC^{n\times r}$ 
is not modified in the manipulations. The size of the basis matrix $Z$
is the same before and after the restart.

We propose two ways of further exploiting the structure of the functions
in order to avoid a dramatic increase in the required memory resources.
\begin{itemize}
 \item Semi--explicit restart (Section~\ref{sec:explicit_restart}): An invariant pair can be completely represented
by exponentials and therefore does not contribute to the memory requirement for $Z$.
The fact that invariant pairs are exponentials was exploited in the restart in \cite{Jarlebring2014Schur}.
We show how the ideas in \cite{Jarlebring2014Schur} can be carried over
to tensor-structured functions. 
More precisely, the adaption of \cite{Jarlebring2014Schur}
involves restarting the iteration with a locked pair, i.e., only the first $p_\ell$ columns of \eqref{eq:Arnoldi_factorization_d}, and a function $f$ constructed in a particular way. 
The approach is outlined in Algorithm~\ref{alg:explicit_restart} with details are specified in section~\ref{sec:implicit_restart}.
 \item Implicit restart (Section~\ref{sec:implicit_restart}): By only representing polynomials, we show that the TIAR-factorization has a particular
structure such that it can be accurately approximated. 
This allows us to carry out a full implicit restart, and subsequently approximate the TIAR-factorization
such that the matrix $Z$ can be reduced in size. 
The adaption is given in Algorithm~\ref{alg:implicit_restart} with  details about the approximation 
specified in section~\ref{sec:explicit_restart}. Step 6 of Algorithm~\ref{alg:implicit_restart}
is given in Algorithm~\ref{alg:approximation_TIAR}.
\end{itemize}

\input{algorithm_explicit_restart}
\input{algorithm_implicit_restart}

%
%
%

%% file: algorithm_explicit_restart.tex
\begin{algorithm}\label{alg:explicit_restart}
\caption{Semi--explicit restarting for TIAR in operator setting \label{alg:imp_rest_infarnoldi}}
\SetKwInOut{Input}{input}\SetKwInOut{Output}{output}
\Input{
A normalized tensor structured function  
represented by
$(Z,W,Y,S)\in\CC^{n\times r}\times\CC^{n \times p}\times\CC^{n\times p}\times\CC^{p\times p}$
and 
$(a,b,C)\in\CC^{ d \times 1 \times r}\times \CC^{d \times 1 \times p}\times\CC^{p \times m}$
}
\Output{
$p$ eigenvalues of $\BBB$ 
}
\nl Set $\Psi^{(1)}=[\psi]$, $H^{(1)}$ empty matrix of size $1 \times 0$ and $j=1$	\\
\BlankLine
\While{$\pl \le p$ }{
\nl Expand the the TIAR factorization $(\Psi^{(j)},\Hul^{(j)})$ to length $m$ using algorithm \ref{alg:expand_TIAR}	\\
\nl Compute the $\pl$ converged Ritz pairs and $P$, $R_{i,j}$ and $a_i$ given in \eqref{eq:Arnoldi_factorization_a} 	\label{step:explicit_pl}\\
\nl Compute the matrices $Q$, $F$, $H$ and $\beta$ given in \eqref{eq:Arnoldi_factorization_d} 				\\
\nl Lock the invariant pair 
    $\bar \Psi = \Psi^{(j)} P I_{k,p_\ell} Q$ and $R_{1,1}$ 	\\
\nl Select $f$ and compute $\bar f$ the orthogonal complement 
    with respect $\bar \Psi$	\label{step:explicit_f}\\
\nl Set $\Psi^{(j+1)} = [\bar \Psi, \bar f]$ 
    and $\Hul^{(j+1)}=\begin{pmatrix}
                    R_{1,1}	\\	0
                   \end{pmatrix}
$ and $j=j+1$	\label{step:explicit_Hul}\\
}
\nl Return the eigenvalues of $R_{1,1}$ \\
\end{algorithm}

%% file: algorithm_implicit_restart.tex
\begin{algorithm}\label{alg:implicit_restart}
\caption{Implicit restart for TIAR in operator setting}
\SetKwInOut{Input}{input}\SetKwInOut{Output}{output}
\Input{
A normalized tensor structured function  
represented by
$(Z,0,0,0)\in\CC^{n\times r}\times\CC^{n \times p}\times\CC^{n\times p}\times\CC^{p\times p}$
and 
$(a,0,0)\in\CC^{ d \times 1 \times r}\times \CC^{d \times 1 \times p}\times\CC^{p \times m}$
}
\Output{
$p$ eigenvalues of $\BBB$.
}
\BlankLine
\nl Set $\Psi^{(1)}=[\psi]$, $H^{(1)}$ empty matrix of size $1 \times 0$ and $j=1$	\label{step:first_function} \\
\While{$\pl \le p$ }{
\nl Expand the the TIAR factorization $(\Psi^{(j)},\Hul^{(j)})$ to length $m$ using algorithm \label{step:expand_TIAR} \ref{alg:expand_TIAR}	\\
\nl Compute the $\pl$ converged Ritz pairs and $P$, $R_{i,j}$ and $a_i$ given in \eqref{eq:Arnoldi_factorization_a} 	\\
\nl Compute the matrices $Q$, $F$, $H$ and $\beta$ given in \eqref{eq:Arnoldi_factorization_d} 				\\
\nl Set 
$ \Psi^{(j+1)} = [\Psi^{(j)} P I_{k,p} Q \ , \ \Psi^{(j)} e_m]$, 
$
 \Hul^{(j+1)} = 
 \begin{pmatrix}
  R_{1,1}	&	F		\\
		&	H		\\
		&	\beta e_{p-\pl}
 \end{pmatrix}
$ \label{step:change_basis}
\\
\nl Approximation of TIAR factorization, algorithm \ref{alg:approximation_TIAR}\label{step:implicit_approx}
}
\nl Return the eigenvalues of $R_{1,1}$\\
\end{algorithm}

%% file: explicit_restart.tex
\section{Tensor structure exploitation for the semi--explicit restart} \label{sec:explicit_restart}
A restarting strategy for IAR, based representing functions as sums of exponentials and polynomials,
was presented in \cite{Jarlebring2014Schur}.
A nice feature of that approach is that the invariant pairs can be exactly
represented, and locking can be efficiently incorporated.
Due to the explicit storage of polynomial coefficients
in  \cite{Jarlebring2014Schur}, the approach still requires considerable
memory. We here show that by representing the functions implicitly
as tensor-structured functions \eqref{eq:structured_functions}
we can maintain the advantages of \cite{Jarlebring2014Schur}
but improve performance (both in memory and CPU-time). 
This construction is equivalent to \cite{Jarlebring2014Schur},
but more efficient.

The expansion of the TIAR factorization with tensor structured functions 
(as described in Algorithm~\ref{alg:expand_TIAR})
combined with the locking procedure (as described in Section~\ref{sec:KS})
results in Algorithm~\ref{alg:explicit_restart}. 
Steps~\ref{step:explicit_pl}-\ref{step:explicit_Hul} follow
the procedure described in \cite{Jarlebring2014Schur} adapted
for tensor-structured functions.
In Step~\ref{step:explicit_f} 
the function used as a new starting function can be extracted
from the tensor structured representation as follows,
completely equivalent with \cite{Jarlebring2014Schur}.
\begin{align*}
 f(\theta) = \tilde Y \exp (\theta S) e_{\pl+1},
& &
S:=\begin{pmatrix}
R_{1,1}	&	F			\\
	&	H			\\
\end{pmatrix}^{-1},
& &
\tilde Y	&: = \Psi_{m} (0) P \ I_{k,p} \ Q.
\end{align*}
We can use Observation \ref{obs:linearity} 
to compute $\tilde Y$ from the tensor structured function representation. 
We define $M:=P \ I_{k,p} \ Q$ such that 
we obtain
\begin{align*} 
\tilde Y	&: = \Psi_{m} (0) M
\\ &=
 P_d(0)
\left(
\sum_{\ell=1}^r 
a_{:,:,\ell} M
\otimes z_\ell   
+ \sum_{\ell=1}^p 
b_{:,:,\ell} M
\otimes w_\ell \right)  
+ Y \exp_{d} (0) C 
\\ &=
\sum_{\ell=1}^r 
a_{1,:,\ell} M
\otimes z_\ell   
+ \sum_{\ell=1}^p 
b_{1,:,\ell} M
\otimes w_\ell
\end{align*}

%% file: implicit_restart.tex
\section{Tensor structure exploitation for the implicit polynomial restart} \label{sec:implicit_restart}

In contrast to the procedure in Section~\ref{sec:explicit_restart}, where the main idea was 
to do locking with exponentials and restart with a factorization of length $p_\ell$, we
now propose a fully implicit procedure involving a factorization of length $p$. 
In this setting we use $Y=0$, i.e., only representing polynomials with the tensor structured functions. 
This allows us to develop theory for the structure of the coefficient matrix,
which can be exploited in an approximation of the TIAR factorization. 
The algorithm is summarized in Algorithm~\ref{alg:implicit_restart}. 

The approximation in Step~\ref{step:implicit_approx} is
done in order to avoid the growth in memory requirements for the representation.
The approximation technique is derived in the following subsections
and summarized in Algorithm~\ref{alg:approximation_TIAR}.

Our approximation approach is based on degree reduction and approximation 
with a truncated singular value decomposition. A compression with a truncated 
singular value decomposition was also made for the compact representations
in CORK \cite{Beeumen_CORK_2015} and TOAR \cite{Kressner_and_Roman_compact_2014}.
In contrast to \cite{Beeumen_CORK_2015,Kressner_and_Roman_compact_2014} 
our specific setting allows to prove bounds
on the error introduced by the approximations (Section~\ref{sec:svd_compression}-\ref{sec:degree_approx}).
We also show the effictiveness by proving a bound on the decay
of the singular values (Section~\ref{sec:svd_decay}).

We first note the following decay in the magnitude of the elements of
the tensor $a$, which are the coefficients representing $\Psi_k$.

\begin{theorem} \label{thm:decay}
Let 
$ Z \in \CC^{n \times p}$ the matrix and
$ a \in \CC^{(k+1) \times (k+1) \times r}$ 
the coefficients that represent the tensor structured function $\Psi_{k+1}$ 
and $\Hul_{k} \in \CC^{k+1 \times k}$ such that that $(\Psi_{k+1}, \Hul_{k})$ is 
a TIAR factorization. Assume that $\psi_1(\theta)$ is a constant function, i.e., 
$a_{i,1,\ell}=0$ if $i>1$. Then 
\begin{align} \label{eq:decay}
\|a_{i,:,:} \| \le \frac{C}{(i-1)!}\textrm{ for }i=1, \dots, k+1,
\end{align}
where $C=\kappa([v, C_{k+1} v, \dots, C_{k+1}])$, $C_{k+1}$ is defined in \cite[equation (29)]{Jarlebring_INFARN_2012} and 
$v=\sum_{\ell=1}^r a_{:,1,\ell} z_\ell$. 
\end{theorem}

\begin{proof}
 Let $\Phi_{k+1}(\theta) = \left( \psi_1(\theta), \BBB \psi_1(\theta), \dots, \BBB^k \psi_1(\theta) \right)$.
 Applying theorem \ref{thm:action_of_B} with $Y=0$, we obtain
 \begin{equation*}
\Phi_{k+1}(\theta)  =  P_{k}(\theta)
\left(
\sum_{\ell=1}^r 
\hat a_{:,:,\ell}
\otimes z_\ell   
\right)
\end{equation*}
where 
  \begin{align*}
 \hat a_{:,:,\ell} := 
  \begin{pmatrix}
   \frac{a_{1,1,\ell}}{0!}		&	\frac{a_{1,2,\ell}}{0!}		&	\frac{a_{1,3,\ell}}{0!}		&	\dots 	&		\frac{a_{1,k+1,\ell}}{0!}		\\
					&	\frac{a_{1,1,\ell}}{1!}		&	\frac{a_{1,2,\ell}}{1!}		&		&		\frac{a_{1,k,\ell}}{1!}			\\
					&					&	\frac{a_{1,1,\ell}}{2!}		&		&		\frac{a_{1,k-1,\ell}}{2!}		\\
					&					&					&	\ddots	&		\vdots					\\
					&					&					&		&		\frac{a_{1,1,\ell}}{(k+1)!}		\\					
  \end{pmatrix}
  \end{align*}
Since $(\Psi_{k+1}, \Hul_{k})$ forms 
a TIAR factorization, it holds $\sspan \left( \Phi_{k+1} \right) = \sspan \left( \Psi_{k+1} \right)$. Therefore it exists an 
invertible matrix $R \in \CC^{(k+1) \times (k+1)}$ such that 
$\Phi_{k+1}R = \Psi_{k+1}$. Using the Observation~\ref{obs:linearity} we have that $a_{:,:,\ell} = \hat a_{:,:,\ell} R$ and 
by submultiplicativity of the euclidean norm we have that for $i=1, \dots, k+1$
\begin{align} \label{eq:thm_decay_interm}
 \| a_{i,:,\ell} \| = \| \hat a_{i,:,\ell} R \| \le  \| \hat a_{i,:,\ell} \| \| R \|. 
\end{align}
Using the structure of $\hat a_{:,:,\ell}$ we have
\begin{align} \label{eq:thm_decay_interm_2}
 \| \hat a_{i,:,\ell} \|^2 
 \le 
 \frac{1}{(i-1)!}
 \sum_{j=i}^{k+1} \hat a_{i,j,\ell}^2
 \le 
  \frac{1}{(i-1)!}
  \sum_{j=1}^{k+1} \hat a_{i,j,\ell}^2
  =\frac{\| \hat a_{1,:,\ell} \|^2 }{(i-1)!}.
\end{align}
Combining \eqref{eq:thm_decay_interm} and \eqref{eq:thm_decay_interm_2} we obtain 
\begin{align*}
  \| a_{i,:,\ell} \| \le \frac{\| \hat a_{1,:,\ell} \|}{(i-1)!} \| R \| = \frac{\| a_{1,:,\ell} R^{-1} \|}{(i-1)!} \le \frac{\| a_{1,:,\ell}\| }{(i-1)!}  \kappa (R) .
\end{align*}
Setting $C:=\kappa (R)$ we obtain \eqref{eq:decay}. It remains to show that 
$C=\kappa([v, C_{k+1} v, \dots, C_{k+1}])$, $C_{k+1}$. Due to the equaivalence of TIAR and IAR 
and the companion matrix interpretation of IAR 
\cite[theorem 6]{Jarlebring_INFARN_2012}, we have that TIAR is equivalent 
to use the Arnoldi method on the matrix $C_{k+1}$ and starting vector $v=\sum_{\ell=1}^r a_{:,1,\ell} z_\ell$. 
More precisely, the relation $\Phi_{k+1}R = \Psi_{k+1}$ can be written in terms of vectors as 
$V R = W$ where the first column of $V$ and $W$ is $v=\sum_{\ell=1}^r a_{:,1,\ell} \otimes z_\ell$ 
and $W=[v, C_{k+1} v, \dots, C_{k+1}]$.

\end{proof}

\begin{observation}
 In the numerical simulations, we observed a very fast decay of the norm of the matrices $\| a_{i,:,:}\|$ with respect $i$. 
 Unfortunately, the condition number of the Krylov matrix $[v, C_{k+1} v, \dots, C_{k+1}^{k+1} v]$ 
 grows at least exponentially with respect $k$. See \cite{Beckermann2000KrylovMatrix} and the therein references. 
 The bound provided by Theorem~\ref{thm:decay} is pessimistic and not sharp; we use it only for theoretical purposes.
\end{observation}

\begin{corollary} \label{cor:decay}
 If $\Psi_k(\theta)$ given in \eqref{eq:structured_functions_block} satisfies 
 $\|a_{i,:,:} \| \le C / (i-1)!$ for $i=1, \dots, d$, then for any matrix $M$, 
 $\tilde \Psi_k(\theta) = \Psi_k(\theta) M$ satisfies 
 $\| \tilde a_{i,:,:} \| \le \tilde C / (i-1)!$ for $i=1, \dots, d$ where 
 $\tilde C \le C \kappa (M)$.
\end{corollary}

Let consider the Algorithm~\ref{alg:implicit_restart} with a constant starting function in Step~\ref{step:first_function}, 
i.e., $\psi(\theta)$ is such that $a_{i,1,\ell}=0$ if $i>1$. 
As consequence of Theorem~\ref{thm:decay}, after expansion of the TIAR factorization in Step~\ref{step:expand_TIAR}, 
we have that the norm of $a_{i,:,:}$ satisfies \eqref{eq:decay}. 
By using the Corollary~\ref{cor:decay}  we obtain that this relation is preserved also after 
the Step~\ref{step:change_basis}, which consist in writing the new TIAR factorization with the wanted Ritz values. 
In conclusion, in the Algorithm~\ref{alg:implicit_restart} the coefficients of Krylov basis $\Psi^{(j)}$ 
always fulfill \eqref{eq:decay}. This allow us to introduce an approximation of the TIAR factorization.

\subsection{Approximation by SVD compression}\label{sec:svd_compression}

Given a TIAR factorization with basis function $\Psi_k$ we show in the following
theorem how we can approximate the basis function with less memory,
more precisely with a smaller $Z$-matrix. The theorem also shows
how this approximation influences the influences the approximation
$\Psi_k$. Moreover, we show that the approximation has a small
impact also on the residual of the TIAR factorization.

\begin{theorem} \label{thm:svd_compression}
Let 
$ a \in \CC^{(d+1) \times k \times r}$, 
$Z \in \CC^{n \times r}$
be the coefficients that represent the tensor structured function 
\eqref{eq:structured_functions} 
and suppose that $(\Psi_{k+1}, \Hul_{k})$ is 
a TIAR factorization.
Suppose
 $\left \{ |z| \le R \right \} \subseteq \Omega$ with $R>1$.
Let  
$A:=[A_1,\ldots,A_d] \in \CC^{ r \times dm}$
be the unfolding of the tensor $a$ in 
the sense that $A_i=(a_{i,:,:})^T$.
Given the singular value decomposition of $A$
\begin{align}
& A=[U_1, U] \diag(\Sigma_1,\Sigma)[V_1^H,\ldots,V_d^H] 	\nonumber	\\
& \Sigma_1=\diag(\sigma_1,\ldots,\sigma_{\tilde r})		\label{eq:svd}	\\	 
& \Sigma=\diag(\sigma_{\tilde r+1},\ldots,\sigma_r),		\nonumber
\end{align}
let
\begin{align}
  \tilde{Z}	:= ZU_1				& , &			\label{eq:svd_TIAR}
  \tilde{A}_i	:= \Sigma_1 V_i^H		&   &	i=1,\ldots,d+1.   
\end{align}
and $\tilde \Psi_{k+1}$ the 
tensor structured function defined 
by the coefficients \\ $\tilde a \in \CC^{(d+1) \times (k+1) \times \tilde r}$ 
and $\tilde Z \in \CC^{n \times \tilde r}$, with $\tilde a_{i,:,:}=\tilde A_i^T$.
Then,
\begin{subequations}
\begin{align}	\label{eq:errV}
   \| \Psi_{k+1}-\tilde{\Psi}_{k+1} \|_T & \le \sqrt{(d+1) (k+1)}\sigma_{\tilde r+1}	\\
   \| \BBB \tilde{\Psi}_{k} - \tilde{\Psi}_{k+1} \Hul_p \|_T & \le \sqrt{k} (C_d+C_s) \sigma_{\tilde r+1}\label{eq:errBV}
\end{align}
\end{subequations}
with
\begin{align*}
C_d &:= \gamma + \log(d+1)  + (d+1) \| \Hul_k \|  \\
C_s &:=  \| M_0^{-1} \| 
\left[
(\gamma + \log(s+1) ) \max_{1 \le i \le s} \| M_i \| +
  \max_{|\lambda|=R} \| M(\lambda) \| 
\right]
\end{align*}
where $\gamma \approx 0.57721$ is the Euler--Mascheroni constant and 
\[
 s:=\min \left \{s\in\NN : \frac{C (d-s)}{R^s} \le \sigma_{\tilde r} \right \}
\]
where $C$ is defined in Corollary~\ref{cor:decay}.
\end{theorem}

\begin{proof}
The proof of \eqref{eq:errV} is based on construction
a difference function $\hat{\Psi}_{k+1}=\Psi_{k+1}-\tilde{\Psi}_{k+1}$
as follows.
We define
\begin{align*}
& \hat Z := Z U,			&&	\hat A_i := \Sigma V_{i}^H,				\\
& X_i := Z A_{i+1}, 			&&	\hat X_i:=\hat Z \hat A_{i+1}, 			&&	\tilde X_i:=\tilde Z \tilde A_{i+1},				\\
& X := [ X_0^H \dots  X_d^H ]^H,	&& 	\hat X := [ \hat X_0^H,\dots, \hat X_d^H ]^H, 	&& 	\tilde X := [ \tilde X_0^H, \dots, \tilde X_d^H ]^H.
\end{align*}
then we can express
$
 \Psi_{k+1} = P_d(\theta) X			
$ 
where 
$
\tilde \Psi_{k+1}(\theta) = P_d(\theta) \tilde X
$
and
$
\hat \Psi_{k+1}(\theta) = P_d(\theta) \hat X
$.
By using \eqref{eq:Pd_frob} and 
$\|\hat{X}_i\|_F^2=\|\hat{Z}\hat{A}_{i+1}\|_F^2\le 
(k+1)\|\hat{Z}\hat{A}_{i+1}\|_2^2=
(k+1)\|\hat{A}_{i+1}\|_2^2=
(k+1)\|\Sigma V_{i+1}\|_2^2\le (k+1)\|\Sigma\|_2^2=(k+1)\sigma_{\tilde{r}+1}^2$
we obtain
\begin{equation*} 
 \| \hat \Psi_{k+1} \|_T^2 
	=	\sum_{i=0}^d \| \hat X_i \|_F^2		
	\le	(d+1) (k+1) \sigma_{\tilde r}^2
\end{equation*}
which proves \eqref{eq:errV}.

In order to show \eqref{eq:errBV} we 
first use that 
$
\| \BBB \tilde \Psi_{k+1} - \tilde \Psi_k \Hul_k \| 
= 
\| \BBB \hat \Psi_{k+1} - \hat \Psi_k  \Hul_k \| 
$
since $(\Psi_{k+1},\Hul_k)$ is a TIAR factorization
and subsequently use the decay of $A_i$ and
analyticity of $M$ as follows.
For notational convenience we define 
\begin{equation}  \label{eq:Yi_def}
Y_i := \hat X_i I_{k+1,k}, \textrm{ for }i=0, \dots, d-1
\end{equation}
and  $Y:=[ Y_0^H \dots Y_d^H ]^H$
such that  we can express
$
 \hat \Psi_k(\theta) = P_{d-1}(\theta) Y
$.

%
%
Using \cite[theorem 4.2]{Jarlebring2014Schur}
for each column of $\hat \Psi_k(\theta)$, we get
$\BBB \hat \Psi_k (\theta) = P_d(\theta) Y_+$ 
with 
\begin{align*}
Y_{+,i+1} 	&:=	\frac{Y_i}{i+1}		
& \mbox{for} &&
i=0, \dots, d-1 
&& \mbox{and} &&
Y_{+,0} 	:= 	-M_0^{-1} \sum_{i=1}^d M_i Y_{+,i}
\end{align*}
By definition and \eqref{eq:Pd_frob} we have
\[
 \| \BBB \hat \Psi_{k} - \hat \Psi_{k+1} \Hul_k \| 
 =\| P_d(\theta) Y_+ - P_d(\theta)  \hat X \Hul_k \| 
 =
 \| Y_+ -  \hat  X \Hul_k \|_F.
\]
Moreover, by using the two-norm bound of the Frobenius norm,
\eqref{eq:Yi_def} and that $\|\hat{X}_i\|\le \sigma_{\tilde{r}+1}$,
\begin{subequations}
\label{eq:residual}
\begin{align}  
 \| Y_+ -  \hat  X \Hul_k \|_F 
 & \le
 \sum_{i=0}^d
 \| Y_{+,i} -  \hat  X_{i} \Hul_k \|_F	
  \le 
 \sqrt{k}  \sum_{i=0}^d ( \| Y_{+,i} \| + \|  \hat  X_{i}  \| \| \Hul_k \|)
\\	 
&
=
\sqrt{k} \left(
\| Y_{+,0} \| +
\sum_{i=1}^d  \| Y_{+,i} \| + \sum_{i=0}^d \|  \hat X_{i}  \| \| \Hul_k \| 
\right)
\\	 
& \le 
\sqrt{k} \left(
\| Y_{+,0} \| + 
\sum_{i=1}^d   \frac{\| \hat  X_{i-1} I_{n,k}\|}{i}   + \sum_{i=0}^d  \|  \hat  X_{i}  \| \| \Hul_k \| 
\right)\\
& \le 
\sqrt{k} \left(
\| Y_{+,0} \| + 
\sum_{i=1}^d   \frac{\sigma_{\tilde{r}+1}}{i}   + \sum_{i=0}^d  \sigma_{\tilde{r}+1} \| \Hul_k \| 
\right) \\
& \le
\sqrt{k} \left[
\| Y_{+,0} \| + 
\sigma_{\tilde r+1} \left( \gamma + \log(d+1)  + (d+1) \| \Hul_k \|  
\right) \right]
\end{align}
\end{subequations}
In the last inequality we use 
the Euler-Mascheroni inequality where $\gamma$ 
is defined in \cite[Formula 6.1.3]{abramowitz1964handbook}.
It remains to bound $\|Y_{+,0}\|$. 
By using the definition of $Y_{+,0}$ and again
applying the Euler-Mascheroni inequality
we have that 
\begin{align} \nonumber
 \| Y_{+,0} \|
 & \le \| M_0^{-1} \| \sum_{i=1}^d \| M_i \| \frac{ \|  \hat X_{i-1} I_{n,k} \|}{i} 
  \le \| M_0^{-1} \| \sum_{i=1}^d \| M_i \| \frac{ \|  \hat X_{i-1} \|}{i}	\\ \nonumber
 & = \| M_0^{-1} \| \left( 
 \sum_{i=1}^s \| M_i \| \frac{ \|  \hat X_{i-1} \|}{i} +
  \sum_{i=s+1}^d \| M_i \| \frac{ \|  \hat X_{i-1} \|}{i} 
 \right) \\ 
  & \le \| M_0^{-1} \| \left( 
\sigma_{\tilde r+1} ( \gamma + \log(s+1) ) \max_{1 \le i \le s} \| M_i \| +
  \sum_{i=s+1}^d \| M_i \| \frac{\| \hat X_{i-1} \|}{i} 
 \right). \label{eq:Y0a}
\end{align}
As consequence of the Cauchy integral formula
\begin{align}
\| M_i \| \frac{ \| \hat X_{i-1} \|}{i}  
\le 
\| M_i \| \frac{\| A_{i} \|}{i} 
\le C \frac{ \| M_i \|}{i!} 
\le C \frac{ \displaystyle \max_{|\lambda|=R} \| M(\lambda) \|}{R^i}. \label{eq:Y0b}
\end{align}
By substituting  \eqref{eq:Y0b} in \eqref{eq:Y0a} we obtain
\begin{align} \nonumber
  \| Y_{+,0} \| 
  &\le 
  \sigma_{\tilde r+1} \| M_0^{-1} \| ( \gamma + \log(s+1) ) \max_{1 \le i \le s} \| M_i \| +
\max_{|\lambda|=R} \| M(\lambda) \| C \frac{ d-s }{R^s} \\
& \le
  \sigma_{\tilde r+1}  \| M_0^{-1} \| 
  \left(
  ( \gamma + \log(s+1) ) \max_{1 \le i \le s} \| M_i \| +
  \max_{|\lambda|=R} \| M(\lambda) \|  
  \right). \label{eq:Y0c}
\end{align}
We reach the conclusion \eqref{eq:errBV}
from the combination of \eqref{eq:Y0c} in \eqref{eq:residual}.

\end{proof}

\subsection{Approximation by reducing the degree} \label{sec:degree_approx}
Another approximation which reduces the storage requirements can
be done by truncating the polynomial in $\Psi_k$. The
following theorem illustrated the approximation properties of this approach.
\begin{theorem} \label{thm:degree}
Let 
$a\in \CC^{(d+1) \times (k+1) \times r}$, 
be the representation of 
 the tensor structured function 
$\Psi_{k+1}$ with $Y=0$.
For $\tilde d \le d$ let  
\begin{equation}
\tilde \Psi_{k+1}(\theta) := 
P_{\tilde d}(\theta)
\left(
\sum_{\ell=1}^r 
\tilde a_{:,:,\ell}
\otimes z_\ell
\right)
\end{equation}
where $\tilde a_{i,j,\ell}=a_{i,j,\ell}$ for $i=1, \dots, \tilde d$, $ j = 1, \dots, k+1$ and $\ell=1, \dots, r$.
Then
\begin{align}
 \| \tilde \Psi_{k+1} - \Psi_{k+1} \| 			& 
 \label{eq:poly_truncation_basis}
  \le C \sqrt{k+1} \frac{(d-\tilde d) }{\tilde d!}	\\
 \| \BBB \tilde \Psi_k - \tilde \Psi_{k+1}\Hul_k  \|	& \le 
   C \sqrt{k+1}  \left( \max_{\tilde d+1 \le i \le d} \| M_i \| \right)
     \| M_0^{-1} \| \frac{d-\tilde d}{(\tilde d+1)!}  
      \label{eq:poly_truncation_residual}
\end{align}

\end{theorem}

\begin{proof}
 We define
 $
 X_i 		:= Z A_{i+1} 	 $ for $	i=0, \dots, d		
 $
 and $ X:=[X_0^T, \dots, X_d^T]$ and $\tilde X:=[X_0^T, \dots,  X_{\tilde d}^T]$
such that 
 $\Psi_{k+1} (\theta) = P_{d} (\theta) X$ and 
 $\tilde \Psi_{k+1} (\theta) = P_{\tilde d} (\theta) \tilde X$. 
 We have
 \begin{align*}
  \| \Psi_{k+1} (\theta) - \tilde \Psi_{k+1} (\theta) \|^2 
  =
  \sum_{i=\tilde d+1}^d \| X_i \|_F^2
  = 
  \sum_{i=\tilde d+1}^d \| A_i \|_F^2 
  \le 
  (k+1)  \sum_{i=\tilde d+1}^d \| A_i \|^2.
 \end{align*}
By using Corollary~\ref{cor:decay} we obtain \eqref{eq:poly_truncation_basis}.

By definition 
$
\Psi_k(\theta) = \Psi_{k+1}(\theta) I_{k+1,k} 
$
and 
$
\tilde \Psi_k(\theta) = \tilde \Psi_{k+1}(\theta) I_{k+1,k} 
$,
using the observation \ref{obs:linearity},
if we define 
$Y_i := X_i I_{k+1,k}$ 
for $i=0, \dots, d-1$ and 
$Y:=[ Y_0^H \dots Y_{d-1}^H ]^H$ and
$\tilde Y:=[ Y_0^H \dots \tilde Y_{\tilde d - 1}^H ]^H$ 
we can express
$
 \Psi_k(\theta) = P_{d-1}(\theta) Y
$ and 
$
 \tilde \Psi_k(\theta) = P_{\tilde d-1}(\theta) \tilde Y.
$
 
 Using \cite[theorem 4.2]{Jarlebring2014Schur}
 for each column of $\Psi_k(\theta)$ and $\tilde \Psi_k(\theta)$, we get
 $\BBB \Psi_k (\theta) = P_d(\theta) Y_+$ and
 $\BBB \tilde \Psi_k (\theta) = P_d(\theta) \tilde Y_+$ 
 with 
 \begin{align*}
 Y_{+,i+1} 	&:=	\frac{Y_i}{i+1}		
 & \mbox{for} &&
 i=0, \dots, d-1 
 && \mbox{and} &&
  Y_{+,0} 	:= 	-M_0^{-1} \sum_{i=1}^d M_i Y_{+,i} \\
\tilde Y_{+,i+1} 	&:=	Y_{+,i+1} 	
 & \mbox{for} &&
 i=0, \dots, \tilde d-1 
 && \mbox{and} &&
 \tilde Y_{+,0} := 	-M_0^{-1} \sum_{i=1}^{\tilde d} M_i  Y_{+,i} \\
 \end{align*}
 
In our notation, the fact that $(\Psi_{k+1}, \Hul_k)$ is a TIAR factorization,
can be expressed as
$P_d(\theta)Y_+=P_d(\theta)X\Hul_k$,
which implies that the monomial coefficients are equal, i.e., 
\begin{equation}  \label{eq:mono_coeff_equal}
Y_{+,i} = X_i \Hul_k\textrm{ for }i=0, \dots, d.
\end{equation}
 Hence,
from \eqref{eq:Pd_frob} we have 
 \begin{align*}
  \| \BBB \tilde \Psi_k - \tilde \Psi_{k+1} \Hul_k \|^2 &=
  \|  P_{d}(\theta) \tilde Y_+- P_{d}(\theta) \tilde X  \Hul_k \|^2 \\
  &=
 \| \tilde Y_+ - \tilde X  \Hul_k \|_F^2 \\
 &= \| \tilde Y_{+,0} -  X_0 \Hul_k \|_F^2 
    + \sum_{i=1}^{\tilde d}  \|  Y_{+,i} -  X_i \Hul_k \|_F^2 
 \\
 & =  \| \tilde Y_{+,0} -  X_0 \Hul_k \|_F^2 
 \end{align*}
In the last step we applied \eqref{eq:mono_coeff_equal}.
Moreover, by again using \eqref{eq:mono_coeff_equal}, we have
 \begin{align*}
  Y_{+,0} - X_0 \Hul_k 	&= -M_0^{-1} \sum_{i=1}^d M_i Y_{+,i} - X_0 \Hul_k	\\
			&= -M_0^{-1} \sum_{i=1}^{\tilde d} M_i \tilde Y_{+,i} 	
			   -M_0^{-1} \sum_{i=\tilde d+1}^d M_i Y_{+,i} 	- X_0 \Hul_k	\\
			&= \tilde Y_{+,0} - X_0 \Hul_k -M_0^{-1} \sum_{i=\tilde d+1}^d M_i \frac{X_{i-1} I_{k+1,k}}{i}.
 \end{align*}
Therefore 
\begin{align*}
 \| \tilde Y_{+,0} - X_0 \Hul_k \| 
 & \le \| M_0^{-1} \| \sum_{i=\tilde d+1}^d \frac{\| M_i \|  \| A_i\|}{i}.
\end{align*}
We obtain \eqref{eq:poly_truncation_residual} by  using the Corollary~\ref{cor:decay}.
\end{proof}

\begin{rmk} \label{rmk:pow_ser_conv}
 The approximation given in Theorem~\ref{thm:degree} can only be effective if 
$ \left( \max_{\tilde d+1 \le i \le d} \| M_i \| \right)/(\tilde d+1)!$ is small. In particular 
this condition is satisfied if the Taylor coefficients $\| M_i \|/i!$ present a fast decay. 
More precisely, this condition correspond to have the coefficients of the power series expansion 
of $M(\lambda)$ that are decaying to zero.
\end{rmk}

\input{algorithm_approximation_TIAR}
\subsection{The fast decay of singular values}\label{sec:svd_decay}

Finally, as a further justification for our approximation procedure, we now
show how fast the singular values decay. The fast
decay in the singular values illustrated below justifies
the effectiveness of the truncation in Section~\ref{sec:svd_compression}.

\begin{lemma} \label{lemma:generation_a}
Let 
$Z \in \CC^{n \times r}, a \in \CC^{d \times (k+1) \times r}$ represent 
the tensor structured function $\Psi_{k+1}$ as in \eqref{eq:structured_functions_block} with $Y=W=0$ 
and let $\Hul_k \in \CC^{(k+1) \times k}$ be a Hessenberg matrix 
such that $(\Psi_{k+1},\Hul_k)$ is TIAR factorization.  
Then, the tensor $a$ is generated by $d$ vectors, in the sense that 
each vector $a_{i,j,:}$ for $i=1, \dots, d$ and 
$j=1, \dots, k$ can be expressed as linear combination of the vectors 
$ a_{i,1,:}$ and $a_{1,k,:}$ for $ i = 1, \dots, k-d $ and $ j = 1, \dots, k$.  
\end{lemma}

\begin{proof}
The proof is based on induction over the length $k$ of the TIAR factorization. The result is trivial if $k=1$. 
Suppose the result holds for some $k$. 
Let $Z \in \CC^{n \times (r-1)}, a \in \CC^{(d-1) \times k \times r}$ represent 
the tensor structured function $\Psi_{k}$ and let $\Hul_{k-1} \in \CC^{k \times (k-1)}$ an Hessenberg matrix 
such that $(\Psi_k,\Hul_{k-1})$ is TIAR factorization.
If we expand the TIAR factorization $(\Psi_k,\Hul_{k-1})$ by using 
the Algorithm~\ref{alg:expand_TIAR}, more precisely by using \eqref{eq:abc_B_a} and \eqref{eq:abc_orth_a}, 
we obtain 
\begin{align*}
 \beta a_{i+1,k+1,:} = 
 \frac{a_{i,k,:}}{i} - 
 \sum_{j=1}^k h_j a_{i,j,:}
 &&
 i=1, \dots, d-1.
\end{align*}
We reach the condition of the theorem by induction.
\end{proof}

\begin{theorem}
Under the same hypothesis of Lemma~\ref{lemma:generation_a}, let $A$ be 
the unfolding of the tensor $a$ in a sense that $A=[A_1, \dots, A_d]$ such that 
$A_i:=(a_{i,:,:})^T$. We have the following decay in the singular values 
\begin{align*}
  \sigma_{i} \le  C \frac{d-R-k+2}{(R-k+1)!} 
 &&
 i=R+1, \dots, d,
\end{align*}
where  $k \le R \le d$ and $C$ is the constant provided by Corollary~\ref{cor:decay}.
\end{theorem}
\begin{proof}
 We define the matrix 
 $\tilde  A:=[ A_1, \dots, A_{ R-k+1}, 0, \dots, 0] \in \CC^{r \times dk}$. 
 Notice that the columns of the matrices $A$ and $\tilde A$ correspond 
 to the vectors $a_{i,j,:}^T$. In particular, using the Lemma~\ref{lemma:generation_a}, we have that 
 $\rk(A_1)=k$ 
 whereas 
 $\rk(A_j)=1$ if $j \le d-k+1$ otherwise $\rk(A_j)=0$.
 Then we have that $\rk(A) = d$ and $\rk( \tilde A) = R$.
 Using Weyl's theorem  
 \cite[Corollary 8.6.2]{golub2012matrix_computations} and Corollary~\ref{cor:decay} we have for $i \geq R + 1$
\end{proof}
\begin{align*}
 \sigma_{i}   & \le \| A - \tilde A \|
\le
 \sum_{ i = R-k+2 }^d \| A_i \|
\le
  \sum_{ i = R-k+2 }^d \frac{C}{(i-1)!}
\le
  C \frac{d-R-k+2}{(R-k+1)!}
\end{align*}

%% file: algorithm_approximation_TIAR.tex
\begin{algorithm}\label{alg:approximation_TIAR}
\caption{Approximation of TIAR factorization}
\SetKwInOut{Input}{input}\SetKwInOut{Output}{output}
\Input{
A TIAR factorization $(\Psi_{\bar k+1},\Hul_{\bar k})$ 
expressed by
$Y, W \in \CC^{n \times p}$, 
$a \in \CC^{d \times \bar k \times r}$, 
$b \in \CC^{d \times \bar k \times p}$ and
$C \in \CC^{p \times \bar k}$ 
}
\Output{
A TIAR factorization $(\Psi_{\bar k+1},\Hul_{\bar k})$ 
expressed by
$Y, W \in \CC^{n \times p}$, 
$a \in \CC^{d \times \bar k \times r}$, 
$b \in \CC^{d \times \bar k \times p}$ and
$C \in \CC^{p \times \bar k}$ 
}
\BlankLine
\nl Compute the SVD decomposition given in \eqref{eq:svd}
    partitioned such that $\tilde \sigma_r \le \varepsilon$							\\
\nl Set $r=\tilde r$, $Z=\tilde Z$, $a_{i,:,:}=\tilde A_i^T$ for $i=1, \dots, d$ given in \eqref{eq:svd_TIAR}	\\
\nl Compute $\tilde d$ such that 
\[
 \left( \max_{\tilde d+1 \le i \le d} \| M_i \| \right)
     \| M_0^{-1} \| \frac{d-\tilde d}{(\tilde d+1)!} <\varepsilon
\]
 \\ 
\nl Reduce the size of the tensor $a_{i,:,:}=a_{i,1:\tilde d,:}$ and set $d=\tilde d$
\end{algorithm}

%% file: complexity.tex
\section{Complexity analysis} \label{sec:complexity}

We presented two different restarting strategies:
the structured semi--explicit restart and the implicit restart. 
They have different performances and in general, one is not 
preferable to the other. The best choice of the restarting strategy 
depends on the problem features. It 
may be convenient to test both methods on the same 
problem.
We now discuss the general performances, in terms of complexity and stability.
The complexity discussion is based on the assumption that the complexity 
of the action of $M_0^{-1}$ 
is neglectable in comparison to the other parts.

\subsubsection*{Complexity of expanding the TIAR factorization}
Independently of which restarting 
strategy is used, 
the main computational effort of the algorithms 
\ref{alg:explicit_restart} and \ref{alg:implicit_restart} 
is the expansion of a TIAR factorization 
described in algorithm \ref{alg:expand_TIAR}. 
The essential computational effort 
of the algorithm \ref{alg:expand_TIAR}
is the computation of $\tilde z$, given in equation \eqref{eq:z_tilde}. 
This operation has complexity $\OOO(drn)$ for each iteration. 
In both restarting strategies $r$ and $d$ are, in general, 
not large due to the way they are automatically selected in the algorithm~\ref{alg:approximation_TIAR}. 

\subsubsection*{Complexity of the restarting strategies}
After an implicit restart 
we obtain a TIAR factorization of length $p$, whereas 
 after a semi--explicit restart, 
we obtain a TIAR factorization of length $\pl$. 
This means that the semi--explicit restart requires 
a re--computation phase, i.e. after the restart we need to perform 
extra $p-\pl$ steps in order to have a TIAR factorization of length $p$. 
If $p-\pl$ is large, i.e. not many Ritz values converged in comparison 
to the restarting parameter $p$, then the re--computation phase is the 
essential computational effort 
of the algorithm. Notice that this is hard to predict since we do not know how fast the Ritz 
values will converge. 

\subsubsection*{Stability of the restarting strategies} \label{sec:stability}
We will illustrate in section \ref{sec:numerical_experiments} that the restarting 
approaches have different stability properties. The semi--explicit restart tends 
to be efficient if only a few eigenvalues are wanted, i.e. if $p$ is small. 
This is due to the fact that we impose the structure in the starting function. 
On the other hand the implicit restart requires a thick restart in order to be stable in 
several situations, see corresponding discussions for the linear case in 
\cite[chapter 8]{lehoucq1995analysis_restart} . Then $p$ has to be large enough in a sense that at each restart 
the $p$ wanted Ritz values have the corresponding residual not small. 
This leads to additional computational and memory resources.

If we use the semi--explicit restart, 
then the computation of $\tilde z$, in equation \eqref{eq:z_tilde}, involves the term 
$\MM_d(Y,S)$. This quantity can be computed in different ways.
In the simulations we must choose between 
 \eqref{eq:nep1} or \eqref{eq:nep2}. The choice 
 influences the stability of the algorithm. 
In particular if one eigenvalue of $S$ is close to $\partial \Omega$ 
and $M(\lambda)$ is not analytic in $\partial \Omega$, the series \eqref{eq:nep2} 
 converges slowly and in practice overflow can occur.  
 In such situations, \eqref{eq:nep1} is preferable. 
 Notice that it is not always it is possible to use
  \eqref{eq:nep1} since many problems cannot be formulated as 
  \eqref{eq:nep_form_f} with small $q$.

\subsubsection*{Memory requirements of the restarting strategies}
From a memory point of view, the essential part 
of the semi--explicit restart is the storage 
of the matrices $Z$ and $Y$, that is 
$\OOO(nm+np)$.
In the implicit restart the essential part 
is the storage of the matrix $Z$ and 
requires $\OOO(nr_{\max})$ 
where $r_{\max}$ denotes the maximum value 
that the variable $r$ takes in the algorithm. 
The size of 
$r_{\max}$ is not predictable since it 
depends on the svd--approximation 
introduced in algorithm \ref{alg:approximation_TIAR}.
Since 
in each iteration of the algorithm \ref{alg:expand_TIAR} 
the variable $r$ is increased, 
it holds $r_{\max} \geq m-p$. 
Therefore, in the optimal case where $r_{\max}$ 
takes the lower value, the two methods are comparable 
in terms of memory requirements. Notice that, 
the semi--explicit restart requires less memory and has the 
advantage that the required memory is problem independent.


%% file: numerical_experiments.tex
\section{Numerical experiments} \label{sec:numerical_experiments}

\subsection{Delay eigenvalue problem}
\label{sec:dep}
In order to illustrate properties of the proposed restart methods and 
advantages in comparison 
to other approaches, we carried out numerical simulations for solving 
the delay eigenvalue problem (DEP). 
More precisely, we consider the DEP associated with the 
delay differential equation defined in \cite[sect 4.2]{jarlebring2015poloni} with $\tau=1$. 
By using a standard second order finite difference discretization, the DEP is formulated as 
\begin{align*}
 M(\lambda) = - \lambda^2 I + \lambda A_1 + A_0 + e^{-\lambda} A_2 + I.
\end{align*}
We show how the proposed methods perform in terms of $m$, 
the maximum length of the TIAR factorization, and $p$, the number of wanted Ritz values. 

Table~\ref{tab:PDDE_explicit_1} 
and Table~\ref{tab:PDDE_explicit_2} show the advantages of our  semi--explicit restart approach 
in comparison to the equivalent method described in \cite{Jarlebring2014Schur}. 
Our new approach is faster in terms of CPU--time and can solve larger problems due to the memory 
efficient representation of the Krylov basis.
 
Table~\ref{tab:PDDE_implicit_1} 
and Table~\ref{tab:PDDE_implicit_2} 
show the  effectiveness of approximations introduced in Section~\ref{sec:svd_compression} 
and \ref{sec:degree_approx} in comparison to the corresponding restart procedure without approximations. 
In particular, in Algorithm~\ref{alg:approximation_TIAR} 
we consider a drop tolerance $\varepsilon=10^{-14}$. Since the DEP is defined by entire functions, 
the power series coefficients decay to zero and, according to Remark~\ref{rmk:pow_ser_conv}, the approximation by reducing the 
degree is expected to be effective.
By approximating the TIAR factorization, 
the implicit restart requires less resources in terms of memory and CPU--time 
and can solve larger problems. 

We now illustrate the differences between the semi--explicit and the implicit restart. 
More precisely, we show how the parameters $m$ and $p$ 
influence the convergence of the Ritz values with respect the number of iterations. 
The convergence 
of the semi--explicit restart appear to be slower in the semi--explicit restart 
when $p$ is not sufficiently large. See Figure~\ref{fig:PDDE1_conv}. 
The convergence speed of both restarting strategies is 
comparable for a larger $m$ and $p$. See Figure~\ref{fig:PDDE2_conv}.

In practice, the performance of the two restarting strategies 
corresponds to a trade-off between CPU--time and memory. In particular, due to the 
fact that we impose the structure, the semi--explicit restart does not have a growth 
in the polynomial part at each restart and therefore requires less memory. On the other hand, 
for this problem, the semi--explicit restart appears to be slower in term of CPU--time. 
See Figure \ref{fig:PDDE1} and \ref{fig:PDDE2}.

\input{figure_PDDE}

\subsection{Waveguide eigenvalue problem}
\label{sec:waveguide}

In order to illustrate 
how the performance depends on the problem properties, we now consider a NEP 
defined by functions with branch point and branch cut singularities. 
More precisely, we consider 
the waveguide eigenvalue problem (WEP) described in 
\cite[Section 5.1]{WAVEGUIDE_ARNOLDI_2015} after the Cayley transformation. 
In this problem, $\Omega$ is the unit disc 
and there are branch point singularities in $\partial \Omega$. 
Thus, due to the slow convergence of the power series, 
in the semi--explicit restart we have to use \eqref{eq:nep2} in order to 
compute $\MM_d(Y,S)$. This also implies that the approximation by reducing the 
degree is not expected to be effective since the power series coefficients of $M(\lambda)$ 
are not decaying to zero.

In analogy to the previous subsection, we carried out numerical simulations 
in order to compare the semi--explicit and the implicit restart. 

With Figure~\ref{fig:waveguide1_cpu} and \ref{fig:waveguide2_cpu}, 
we illustrate the performance of the 
two restarting approaches with respect the choice of the parameters $m$ and $p$. 
When $p$ is sufficiently large, the residual in 
the semi--explicit restart appears to stagnate after the first restart whereas  
it decreases in a regular way in the implicit restart. 
See Figure \ref{fig:waveguide1_cpu}. 
On the other hand, when $p$ is small, 
the behavior of the residual appear to be specular. 
See Figure \ref{fig:waveguide2_cpu}.
This is due to the fact that semi--explicit restart imposes the structure 
on $p$ vectors which is not beneficial when they do not contain eigenvector approximations. 

It is known that this specific problem has two eigenvalues. 
Therefore, in order to reduce the CPU--time and the memory resources, 
the the number of wanted Ritz values $p$ should be selected small. 
As consequence of the above discussion, we conclude that 
the semi--explicit restart is the best restarting strategy for this problem.

\input{figure_waveguide}

%% file: figure_PDDE.tex
\begin{figure}[htb]
  \centering
  \subfloat[Convergence \label{fig:PDDE1_conv}]{%
  \begin{minipage}[b]{0.49\textwidth}
    \includegraphics{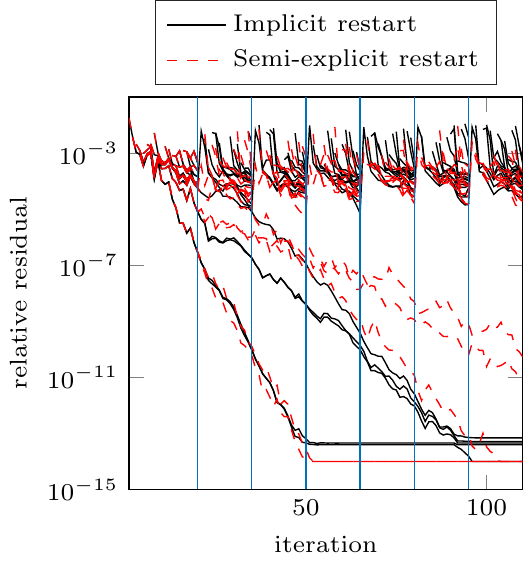}
  \end{minipage}}
  \hfill
  \subfloat[Memory \label{fig:PDDE1_mem}]{%
  \begin{minipage}[b]{0.49\textwidth}
    \includegraphics{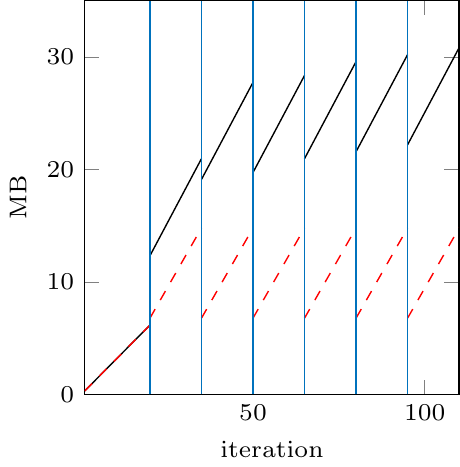}
  \end{minipage}}
  \caption{Implicit and semi--explicit restart for DEP of size $n=40401$ with $m=20$, $p=5$ and restart=7} \label{fig:PDDE1}
\end{figure}

\begin{figure}[htb]
  \centering
  \subfloat[Convergence \label{fig:PDDE2_conv}]{%
  \begin{minipage}[b]{0.49\textwidth}
    \includegraphics{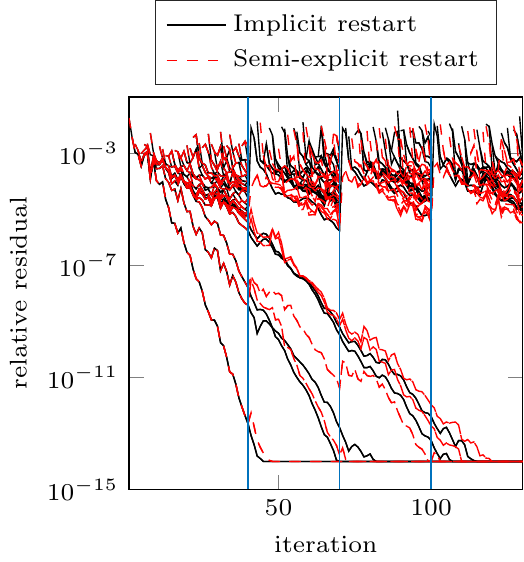}
  \end{minipage}}
  \hfill
  \subfloat[Memory \label{fig:PDDE2_mem}]{%
  \begin{minipage}[b]{0.49\textwidth}
    \includegraphics{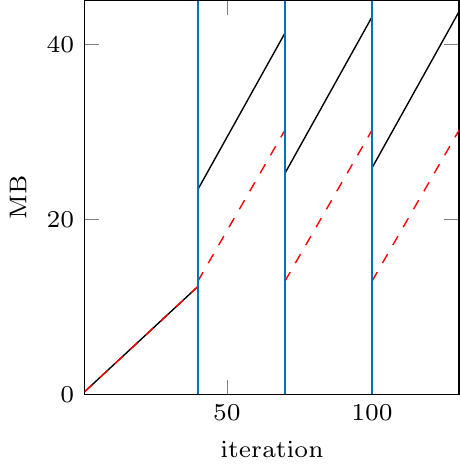}
  \end{minipage}}
  \caption{Implicit and semi--explicit restart for DEP  of size $n=40401$ with $m=40$, $p=10$ and restart=4}  \label{fig:PDDE2}
\end{figure}

\begin{table}[htb]
  \small
  \centering
  \subfloat[$m=20$, $p=5$, restart=7]{%
\begin{tabular}{|c|c|c|c|c|}
   \cline{2-5}
  \multicolumn{1}{c}{}		&	 \multicolumn{4}{|c|}{Semi--explicit restart} 		\\ 
   \cline{2-5}
  \multicolumn{1}{c}{}		&	 \multicolumn{2}{|c|}{tensor structured functions} 	&	\multicolumn{2}{|c|}{ original approach \cite{Jarlebring2014Schur} }   		\\ 
 \hline
 Size	     		&	CPU		&	Memory 		&	CPU		&	Memory 		\\	
 \hline  
	10201		&	19.07s		&	3.73 MB		&	31.41 s		&	65.38 MB	\\	
 \hline  
	40401		&	30.14s		&	14.80 MB	&	1m30s		&	258.92 MB	\\	
 \hline   
	160801		&	1m47s		&	58.89 MB	&	6m04		&	1.01 GB		\\	
 \hline    
 	641601		&	7m30s		&	234.96 MB	&	24m27s		&	4.02 GB		\\	
 \hline    
  	1002001		&	12m01s		&	366.94 MB	&		-	&	-	\\	
 \hline    
\end{tabular}  \label{tab:PDDE_explicit_1}
  }\vspace{1cm}
    \centering
  \subfloat[$m=40$, $p=10$, restart=4]{%
\begin{tabular}{|c|c|c|c|c|}
   \cline{2-5}
  \multicolumn{1}{c}{}		&	 \multicolumn{4}{|c|}{Semi--explicit restart} 		\\ 
   \cline{2-5}
  \multicolumn{1}{c}{}		&	 \multicolumn{2}{|c|}{tensor structured functions} 	&	\multicolumn{2}{|c|}{ original approach \cite{Jarlebring2014Schur} }   		\\ 
 \hline
 Size	     		&	CPU		&	Memory	&	CPU		&	Memory 		\\	
 \hline  
	10201		&	13.47s		&	7.62 MB &	1m05s		&	255.27 MB 	\\	
 \hline  
	40401		&	41.81s		&	30.20 MB &	4m		&	1 GB		\\	
 \hline   
	160801		&	144.79s		&	120.23 MB &	15m54s		&	3.93 GB		\\	
 \hline    
	641601		&	10m43s		&	479.71 MB &	-		&	-	\\	
 \hline  
  	1002001		&	16m21s		&	749.18 MB  &	-		&	-	\\	
 \hline    
\end{tabular}  \label{tab:PDDE_explicit_2}
  } 
    \caption{Semi--explicit restart for DEP.}
\end{table}


\begin{table}[htb]
  \centering
  \subfloat[$m=20$, $p=5$, restart=7]{%
\begin{tabular}{|c|c|c|c|c|}
   \cline{2-5}
  \multicolumn{1}{c}{}		&	 \multicolumn{4}{|c|}{Implicit restart} 		\\ 
   \cline{2-5}
  \multicolumn{1}{c}{}		&	 \multicolumn{2}{|c|}{compression} 	&	\multicolumn{2}{|c|}{no compression}   		\\ 
   \hline 
 Problem size  			&	CPU		&	Memory	&	CPU		&	Memory 		\\	
 \hline 
	10201		&	6.82s		&	7.78 MB		&	11.95s		&	17.12 MB	\\	
 \hline  
	40401		&	21.96s		&	30.82 MB	&	37.63s		&	67.81 MB	\\	
 \hline   
	160801		&	1m20s		&	120.23 MB	&	2m21s		&	269.90 MB	\\	
 \hline    
 	641601		&	5m24s		&	469.92 MB	&	9m33s		&	1.05 GB		\\	
 \hline    
  	1002001		&	8m36s		&	733.89 MB	&	15m16s		&	1.64 GB		\\	
 \hline    
\end{tabular}   \label{tab:PDDE_implicit_1}
  }\vspace{1cm}
    \centering
  \subfloat[$m=40$, $p=10$, restart=4]{%
\begin{tabular}{|c|c|c|c|c|}
   \cline{2-5}
  \multicolumn{1}{c}{}		&	 \multicolumn{4}{|c|}{Implicit restart} 		\\ 
   \cline{2-5}
  \multicolumn{1}{c}{}		&	 \multicolumn{2}{|c|}{compression} 	&	\multicolumn{2}{|c|}{no compression}   		\\ 
 \hline 
 Problem size  			&	CPU		&	Memory	&	CPU		&	Memory 		\\	
 \hline  
	10201			&	9.54s		&	11.05 MB &	16.61s		&	20.24 MB 	\\	
 \hline  
	40401			&	30.48s		&	43.76 MB &	50.66s		&	80.14 MB	\\	
 \hline   
	160801			&	1m54s		&	174.21 MB &	3m11s		&	318.97 MB	\\	
 \hline    
	641601			&	8m05s		&	695.09 MB &	13m14s		&	1.24 GB		\\	
 \hline  
  	1002001			&	12m17s		&	1.06 GB  &	20m57s		&	1.94 GB		\\	
 \hline    
\end{tabular} \label{tab:PDDE_implicit_2}
  }
    \caption{Implicit restart for the DEP.}
\end{table}

%% file: figure_waveguide.tex
\begin{figure}[htb]
  \centering
  \subfloat[Convergence \label{fig:waveguide1_cpu}]{%
  \begin{minipage}[b]{0.49\textwidth}
    \includegraphics{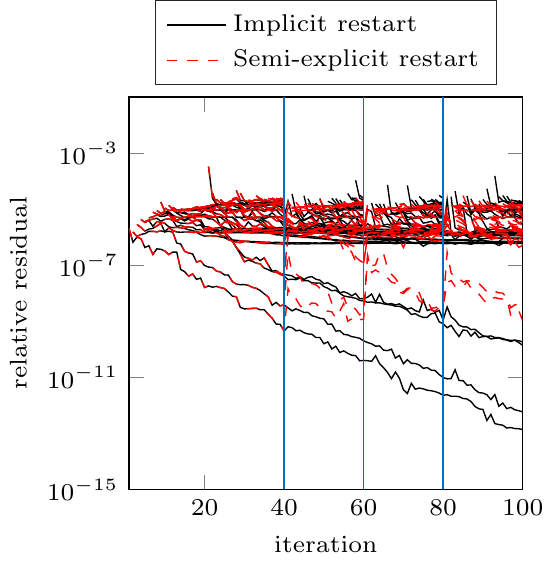}
  \end{minipage}}
  \hfill
  \subfloat[Memory \label{fig:waveguide1_mem}]{%
  \begin{minipage}[b]{0.49\textwidth}
    \includegraphics{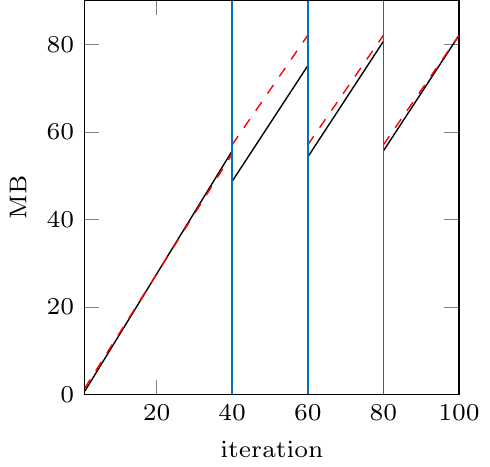}
  \end{minipage}}
  \caption{Implicit and semi--explicit restart for WEP of size $n=40803$ with $m=40$, $p=20$ and restart=4}
\end{figure}

\begin{figure}[htb]
  \centering
  \subfloat[Convergence \label{fig:waveguide2_cpu}]{%
  \begin{minipage}[b]{0.49\textwidth}
    \includegraphics{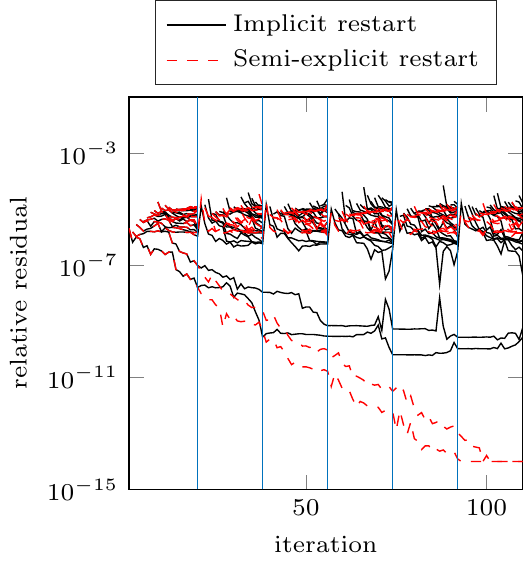}
  \end{minipage}}
  \hfill
  \subfloat[Memory \label{fig:waveguide2_mem}]{%
  \begin{minipage}[b]{0.49\textwidth}
    \includegraphics{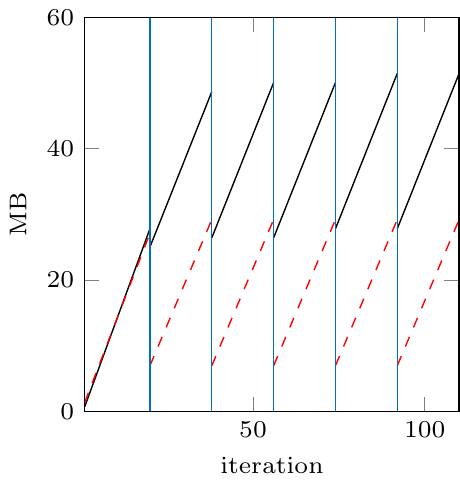}
  \end{minipage}}
  \caption{Implicit and semi--explicit restart for WEP of size $n=91203$ with $m=20$, $p=4$ and restart=6}
\end{figure}

\begin{table}[htb]
  \small
  \centering
  \subfloat[$m=20$, $p=4$, restart=6]{%
  \begin{tabular}{|c|c|c|c|c|}
   \cline{2-5}
  \multicolumn{1}{c}{}		&	 \multicolumn{2}{|c|}{Implicit} 		&	\multicolumn{2}{|c|}{Semi--explicit}   		\\ 
 \hline
 Size	     		&	CPU		&	Memory	&	CPU		&	Memory 		\\	
 \hline  
2703			&	5.33s		&	2.43 MB	 &	13.02s		&	2.43 MB		\\	
 \hline  
 10403			&	9.78s		&	9.36 MB	 &	22.54s		&	9.36 MB		\\	
 \hline   
40803			&	29.15s		&	36.73 MB &	1m14s		&	36.73 MB	\\	
 \hline    
 161603			&	1m50s		&	143.02 MB &	4m01s		&	145.48 MB	\\	
 \hline   
  643203		&	7m27s		&	579.05 MB &	17m44s		&	579.05 MB	\\	
 \hline    
 1006009		&	11m54s		&	903.87 MB &	27m55s		&	903.87 MB	\\	
  \hline
\end{tabular}
\label{tab:waveguide1}  
  }\vspace{1cm}
    \centering
  \subfloat[$m=40$, $p=20$, restart=4]{%
\begin{tabular}{|c|c|c|c|c|}
   \cline{2-5}
  \multicolumn{1}{c}{}		&	 \multicolumn{2}{|c|}{Implicit} 		&	\multicolumn{2}{|c|}{Semi--explicit}   		\\ 
 \hline
 Size	     		&	CPU		&	Memory		&	CPU		&	Memory 		\\	
 \hline  
 2703			&	5.09s		&	1.52 MB	 	&	5.17s		&	0.95 MB		\\	
 \hline  
 10403			&	8.83s		&	5.87 MB		&	10.74s		&	3.65 MB		\\	
 \hline   
 40803			&	25.93s		&	23.04 MB  	&	24.11s		&	14.32 MB	\\	
 \hline    
 161603			&	1m35s		&	91.24 MB  	&	1m20s		&	56.71 MB	\\	
 \hline    
  643203		&	6m31s		&	363.14 MB 	&	5m44s		&	225.73 MB	\\	
 \hline    
1006009			&	10m25s		&	566.83 MB 	&	8m57s		&	352.36 MB	\\	
  \hline
 \end{tabular}  \label{tab:waveguide2}  
  }
    \caption{Implicit and semi--explicit restart for the waveguide problem.}
\end{table}

%% file: conclusion.tex
\section{Concluding remarks and outlook}

In this work we have derived an extension of the TIAR algorithm 
and two restarting strategies. Both restarting 
strategies are based on approximating the TIAR factorization. 
In other works on the IAR--method it has been proven that the basis matrix contains 
a structure that allows exploitations, e.g. for NEPs with low rank structure in the coefficients \cite{van2016rank}. 
An investigation about the combination of 
the approximations of 
the TIAR factorization with such structures of the NEP seems possible 
but deserve further attention. 

Although the framework of TIAR and restarted TIAR is general, 
a specialization 
of the methods to the NEP is required in order to efficiently 
solve the problem. More precisely, an efficient computation 
procedure for computing \eqref{eq:z_tilde} is required. This 
is a nontrivial task for many application and requires problem 
specific research.

%% file: main.bbl
\begin{thebibliography}{10}
\expandafter\ifx\csname url\endcsname\relax
  \def\url#1{\texttt{#1}}\fi
\expandafter\ifx\csname urlprefix\endcsname\relax\def\urlprefix{URL }\fi

\bibitem{abramowitz1964handbook}
M.~Abramowitz, I.~Stegun, Handbook of mathematical functions: with formulas,
  graphs, and mathematical tables, vol.~55, Courier Corporation, 1964.

\bibitem{bai2000templates}
Z.~Bai, J.~Demmel, J.~Dongarra, A.~Ruhe, H.~{van der Vorst}, Templates for the
  solution of algebraic eigenvalue problems: a practical guide, vol.~11, Siam,
  2000.

\bibitem{bai2005soar}
Z.~Bai, Y.~Su, Soar: A second-order {Arnoldi} method for the solution of the
  quadratic eigenvalue problem, SIAM J. Matrix Anal. Appl. 26~(3) (2005)
  640--659.

\bibitem{Beckermann2000KrylovMatrix}
B.~Beckermann, The condition number of real vandermonde, {Krylov} and positive
  definite {Hankel} matrices, Numer. Math. 85~(4) (2000) 553--577.

\bibitem{Beeumen_CORK_2015}
R.~V. Beeumen, K.~Meerbergen, W.~Michiels, Compact rational {Krylov} methods
  for nonlinear eigenvalue problems, SIAM J. Matrix Anal. Appl. 36~(2) (2015)
  820--–838.

\bibitem{betcke2008restarting}
M.~M. Betcke, H.~Voss, Restarting projection methods for rational eigenproblems
  arising in fluid-solid vibrations, Math. Model. Anal. 13~(2) (2008) 171--182.

\bibitem{betcke2016restarting}
M.~M. Betcke, H.~Voss, Restarting iterative projection methods for {Hermitian}
  nonlinear eigenvalue problems with minmax property, Numer. Math. {} (2016)
  1--34.

\bibitem{NLEP_COLLECTION_2011}
T.~Betcke, N.~J. Higham, V.~Mehrmann, C.~Schr{\"o}der, F.~Tisseur, {NLEVP}: A
  collection of nonlinear eigenvalue problems, Tech. rep., Manchester Institute
  for Mathematical Sciences (2011).

\bibitem{betcke2004jacobi_davidson}
T.~Betcke, H.~Voss, A {Jacobi}--{Davidson}-type projection method for nonlinear
  eigenvalue problems, Future Gener. Comp. Sy. 20~(3) (2004) 363--372.

\bibitem{effenberger2013robust}
C.~Effenberger, Robust solution methods for nonlinear eigenvalue problems,
  Ph.D. thesis, {\'E}cole polytechnique f{\'e}d{\'e}rale de Lausanne (2013).

\bibitem{golub2012matrix_computations}
G.~H. Golub, C.~{Van Loan}, F.~Charles, Matrix computations, vol.~3, JHU Press,
  2012.

\bibitem{guttel2014nleigs}
S.~G{\"u}̈ttel, R.~{Van Beeumen}, K.~Meerbergen, W.~Michiels, {NLEIGS}: A
  class of fully rational {Krylov} methods for nonlinear eigenvalue problems,
  SIAM J. Sci. Comput. 36~(6) (2014) A2842--A2864.

\bibitem{Jarlebring2014Schur}
E.~Jarlebring, K.~Meerbergen, W.~Michiels, Computing a partial {Schur}
  factorization of nonlinear eigenvalue problems using the infinite {Arnoldi}
  method, SIAM J. Matrix Anal. Appl. 35~(2) (2014) 411--436.

\bibitem{WAVEGUIDE_ARNOLDI_2015}
E.~Jarlebring, G.~Mele, O.~Runborg, The waveguide eigenvalue problem and the
  tensor infinite {Arnoldi} method, Tech. rep., arXiv:1503.02096 (2015).

\bibitem{Jarlebring_INFARN_2012}
E.~Jarlebring, W.~Michiels, K.~Meerbergen, A linear eigenvalue algorithm for
  the nonlinear eigenvalue problem, Numer. Math. 122~(1) (2012) 169--195.

\bibitem{jarlebring2015poloni}
E.~Jarlebring, F.~Poloni, Iterative methods for the delay {Lyapunov} equation
  with {T-Sylvester} preconditioning, Tech. rep., arXiv:1507.02100 (2015).

\bibitem{kressner2009blockNewton}
D.~Kressner, A block {Newton} method for nonlinear eigenvalue problems, Numer.
  Math. 114~(2) (2009) 355--372.

\bibitem{Kressner_and_Roman_compact_2014}
D.~Kressner, J.~E. Roman, Memory-efficient {Arnoldi} algorithms for
  linearizations of matrix polynomials in {Chebyshev} basis, Numer. Linear
  Algebra Appl. 21~(4) (2014) 569--588.

\bibitem{lancaster2005pseudospectra}
P.~Lancaster, P.~Psarrakos, On the pseudospectra of matrix polynomials, SIAM J.
  Matrix Anal. Appl. 27~(1) (2005) 115--129.

\bibitem{lehoucq1995analysis_restart}
R.~B. Lehoucq, Analysis and implementation of an implicitly restarted {Arnoldi}
  iteration, Ph.D. thesis, Rice University (1995).

\bibitem{Sorensen_IRA_1996}
R.~B. Lehoucq, D.~C. Sorensen, Deflation techniques for an implicitly restarted
  {Arnoldi} iteration, SIAM J. Matrix Anal. Appl. 17~(4) (1996) 789--821.

\bibitem{mehrmann2006structured_PEP}
D.~S. Mackey, N.~Mackey, C.~Mehl, V.~Mehrmann, Structured polynomial eigenvalue
  problems: Good vibrations from good linearizations, SIAM J. Matrix Anal.
  Appl. 28~(4) (2006) 1029--1051.

\bibitem{mackey2015_PEP}
D.~S. Mackey, N.~Mackey, F.~Tisseur, Polynomial eigenvalue problems: Theory,
  computation, and structure, in: Numerical Algebra, Matrix Theory,
  Differential-Algebraic Equations and Control Theory, Springer, 2015, pp.
  319--348.

\bibitem{meerbergen2001locking}
K.~Meerbergen, Locking and restarting quadratic eigenvalue solvers, SIAM J.
  Sci. Comput. 22~(5) (2001) 1814--1839.

\bibitem{meerbergen2008quadratic}
K.~Meerbergen, The quadratic {Arnoldi} method for the solution of the quadratic
  eigenvalue problem, SIAM J. Matrix Anal. Appl. 30~(4) (2008) 1463--1482.

\bibitem{Mehrmann2004nonlinear}
V.~Mehrmann, H.~Voss, Nonlinear eigenvalue problems: A challenge for modern
  eigenvalue methods, GAMM-Mitt. 27~(2) (2004) 121--152.

\bibitem{Morgan96RestartArnoldi}
R.~Morgan, On restarting the {Arnoldi} method for large nonsymmetric eigenvalue
  problems, Math. Comp. 65~(215) (1996) 1213--1230.

\bibitem{neumaier1985residual}
A.~Neumaier, Residual inverse iteration for the nonlinear eigenvalue problem,
  SIAM J. Numer. Anal. 22~(5) (1985) 914--923.

\bibitem{Stewart2002KrylovSchur}
G.~W. Stewart, A {Krylov}--{Schur} algorithm for large eigenproblems, SIAM J.
  Matrix Anal. Appl. 23~(3) (2002) 601--614.

\bibitem{su2011rep}
Y.~Su, Z.~Bai, Solving rational eigenvalue problems via linearization, SIAM J.
  Matrix Anal. Appl. 32~(1) (2011) 201--216.

\bibitem{szyld2015Hermitian_part2}
D.~Szyld, E.~Vecharynski, F.~Xue, Preconditioned eigensolvers for large-scale
  nonlinear {Hermitian} eigenproblems with variational characterizations. {II}.
  {Interior} eigenvalues, SIAM J. Sci. Comput. 37~(6) (2015) A2969--A2997.

\bibitem{szyld2015Hermitian_part1}
D.~Szyld, F.~Xue, Preconditioned eigensolvers for large-scale nonlinear
  {Hermitian} eigenproblems with variational characterizations. {I}. {Extreme}
  eigenvalues, Math. Comp. {} (2016) {}.

\bibitem{tisseur2001quadratic}
F.~Tisseur, K.~Meerbergen, The quadratic eigenvalue problem, SIAM Rev. 2 (2001)
  235--286.

\bibitem{van2016rank}
R.~{Van Beeumen}, E.~Jarlebring, W.~Michiels, A rank-exploiting infinite
  {Arnoldi} algorithm for nonlinear eigenvalue problems, Numer. Linear Algebra
  Appl. {} (2016) {}.

\bibitem{voss2003maxmin}
H.~Voss, A maxmin principle for nonlinear eigenvalue problems with application
  to a rational spectral problem in fluid-solid vibration, Appl. Math. 48~(6)
  (2003) 607--622.

\bibitem{voss2004arnoldi}
H.~Voss, An {Arnoldi} method for nonlinear eigenvalue problems, BIT 44~(2)
  (2004) 387--401.

\bibitem{Voss_2013_NEPCHAPTER}
H.~Voss, Nonlinear eigenvalue problems, in: L.~Hogben (ed.), Handbook of Linear
  Algebra, Second Edition, No. 164 in Discrete Mathematics and Its
  Applications, Chapman and Hall/CRC, 2013.

\bibitem{zhang2013memory}
Y.~Zhang, Y.~Su, A memory-efficient model order reduction for time-delay
  systems, BIT 53~(4) (2013) 1047--1073.

\end{thebibliography}
